\documentclass[12pt]{amsart}
\usepackage{amsmath}
\usepackage{amsxtra}
\usepackage{amstext}
\usepackage{amssymb}
\usepackage{amsthm}
\usepackage{latexsym}
\usepackage{dsfont} 
\usepackage{verbatim}
\usepackage{tabls}
\usepackage{rotating}

\theoremstyle{plain}
\newtheorem{thm}{Theorem}[section]
\newtheorem{cor}[thm]{Corollary}
\newtheorem{lem}[thm]{Lemma}
\newtheorem{prop}[thm]{Proposition}
\newtheorem{defn}[thm]{Definition}

\theoremstyle{definition}
\newtheorem{rem}[thm]{Remark}

\numberwithin{equation}{section}
\def\Xint#1{\mathchoice
   {\XXint\displaystyle\textstyle{#1}}%
   {\XXint\textstyle\scriptstyle{#1}}%
   {\XXint\scriptstyle\scriptscriptstyle{#1}}%
   {\XXint\scriptscriptstyle\scriptscriptstyle{#1}}%
   \!\int}
\def\XXint#1#2#3{{\setbox0=\hbox{$#1{#2#3}{\int}$}
     \vcenter{\hbox{$#2#3$}}\kern-.5\wd0}}

\def\dashint{\Xint-}

\begin{document}

\title[Quasilinear elliptic equations]
{Quasilinear elliptic equations and weighted Sobolev-Poincar\'{e} inequalities with distributional weights}

\author[B.~J. Jaye]
{Benjamin ~J. Jaye}
\address{Department of Mathematics,
Kent State University,
Kent, OH 44240, USA}
\email{bjaye@kent.edu}

\author[V.~G. Maz'ya]
{Vladimir ~G. Maz'ya}

\address{
Department of Mathematics,
Link\"oping University, SE-581 83, Link\"oping,
Sweden}
\email{vladimir.mazya@liu.se}

\author[I.~E. Verbitsky]
{Igor ~E. Verbitsky}
\address{Department of Mathematics,
University of Missouri,
Columbia, MO 65211, USA}
\email{verbitskyi@missouri.edu}

\begin{abstract}  We introduce a class of weak solutions to the quasilinear equation $-\Delta_p u = \sigma |u|^{p-2}u$ in an open set $\Omega\subset\mathbf{R}^n$ with $p>1$, where $\Delta_p u= \nabla \cdot (| \nabla u|^{p-2} \nabla u)$ is the $p$-Laplacian operator.  Our notion of solution is tailored to general distributional coefficients $\sigma$ which satisfy the inequality
\begin{equation*}\label{introeq}-\Lambda\int_{\Omega}|\nabla h|^p dx\leq \langle |h|^p, \sigma \rangle  \leq \lambda \int_{\Omega}|\nabla h|^p dx,\end{equation*}
for all $h\in C^{\infty}_0(\Omega)$.  Here $0<\Lambda< +\infty$, and
 \begin{equation*}
0<\lambda<(p-1)^{2-p} \quad \text{if} \quad p\geq 2, \quad \text{or}  \quad 0< \lambda<1 \quad
\text{if} \quad 0<p<2.
\end{equation*}
%
%

As we shall demonstrate, these conditions on $\lambda$ are natural for the existence of positive solutions, and cannot be relaxed in general.  Furthermore, our class of solutions possesses the optimal local Sobolev regularity available under such a mild restriction on $
\sigma$.

 We also study weak solutions of the closely related equation $-\Delta_p v = (p-1)|\nabla v|^p + \sigma$, under the same conditions on $\sigma$.  Our  results for this latter equation will allow us to characterize the class of $\sigma$ satisfying the above inequality for positive $\lambda$ and $\Lambda$, thereby extending earlier results on the form boundedness problem for the Schr\"odinger operator
 to $p\neq 2$.
\end{abstract}





\thanks{The first and third authors were supported in part
by NSF grant  DMS-0901550.}

\maketitle

\section{Introduction}
This paper concerns a study of weak solutions to certain quasilinear elliptic equations,
and closely related integral inequalities with distributional weights.  Let $p\in (1, \infty)$ and let $\Omega\subset \mathbf{R}^n$ be an open set, with $n\geq 1$. The model equation we consider is given by
\begin{equation}\label{schro}
-\text{ div}(|\nabla u|^{p-2}\nabla u) = \sigma |u|^{p-2}u \, \, \text{ in } \Omega.
\end{equation}
Here $\sigma$ is a distribution which lies in the \text{local} dual Sobolev space $L^{-1,p'}_{\text{loc}}(\Omega)$, where $L^{-1,p'}(\Omega) = L^{1,p}_0(\Omega)^*$ is the dual of the energy space $L^{1,p}_0(\Omega)$ defined as the completion of $C^{\infty}_0(\Omega)$ in the semi-norm $||\nabla (\, \cdot\,)||_{L^p(\Omega)}$  (see Section \ref{prelim} for definitions).  The sole condition we impose on $\sigma$ is the validity of the following weighted Sobolev-Poincar\'{e} inequality:
\begin{equation}\label{pformbd}
|\langle |h|^p, \sigma \rangle | \leq C \int_{\Omega}|\nabla h|^p dx, \text{ for all } h\in C^{\infty}_0(\Omega).
\end{equation}
Here we define the pairing $\langle \cdot, \cdot \rangle$ as follows:  since $h\in C^{\infty}_0(\Omega)$, there is a bounded open set $U$ containing the support of $h$.  Then $\sigma \in L^{-1,p'}(U)$, and $\langle |h|^p, \sigma \rangle$ is the natural dual pairing of $|h|^p\in L_{0}^{1,p}(U)$ and $\sigma$.  
A more concrete realization of this pairing is described in Section \ref{notation}.

We will see that there is a two way correspondence between the inequality (\ref{pformbd}) and the existence  of positive solutions to (\ref{schro}) belonging to a certain class.  Furthermore, our class of weak solutions has the optimal local Sobolev regularity  under the condition (\ref{pformbd}).

  In the case $p=2$, the equation (\ref{schro}) reduces to the time independent Schr\"{o}dinger equation, and condition (\ref{pformbd}) becomes the \textit{form boundedness} property of the potential $\sigma$ (see \cite{RS75}, Sec. X.2).  Even in this linear framework our results are very recent; they were obtained in \cite{JMV11} where a further discussion can be found. The current paper contains a complete
   counterpart for quasilinear operators of the primary results of \cite{JMV11}. This extension is  by no means immediate, and many new ideas are required to handle the non-linear case.

In comparison with the existing literature for (\ref{schro}), the contribution of this paper is that \textit{no additional compactness conditions will be imposed on $\sigma$.}   In particular, we are interested in developing a theory of positive solutions for (\ref{schro}) under such conditions on $\sigma$ so that standard variational techniques do not appear to be applicable.  Furthermore, we do not separate the positive and negative parts of $\sigma$, and hence we will permit highly oscillating potentials,
along with strong singularities, in what follows.

The equation (\ref{schro}) has been attacked by a variety of techniques.  For instance, Smets \cite{Sme99} developed a suitable notion of concentrated compactness (building on the work of P. L. Lions) to study (\ref{schro}).  In order to carry this out, an additional hypothesis beyond (\ref{pformbd}) is imposed on $\sigma$.    A second method we mention is an adaptation of the
methods of Br\'{e}zis and Nirenberg \cite{BN83} by Br\'{e}zis, Marcus and Shafrir, see \cite{BM97,MS00}, in order to study (\ref{schro}) with Hardy-type potentials $\sigma(x) = \text{dist}(x,\partial\Omega)^{-p}$ in a bounded domain $\Omega$.   This second approach makes use of the local compactness properties of $\sigma$ in a profound way.  More recently, a generalization to quasilinear operators of the Allegretto-Piepenbrink theorem for the Schr\"{o}dinger operator has been carried out by Pinchover and Tintarev \cite{PT07, Pin07}.  For additional interesting works on the equation (\ref{schro}), see \cite{AFT04, SW99} and references therein.

These approaches show the subtleties contained in the condition (\ref{pformbd}) in general.   On one hand we do not have local compactness, and on the other hand there is no global dual Sobolev condition contained in (\ref{pformbd}).  It is known  that under the condition (\ref{pformbd}), the equation (\ref{schro}) may display some of the characteristics found in equations with critical Sobolev exponents.  This was observed by Tertikas \cite{Ter98} in the classical case $p=2$.

In this paper, we do not attempt to adapt tools developed for elliptic problems with critical exponents.  Instead, our approach hinges on obtaining quantitative information on the doubling properties of a sequence of solutions to equations which approximate (\ref{schro}).  The argument owes most to the regularity theory of elliptic equations with measure data, in particular the paper of Mingione \cite{Min07}.  We describe our method in more detail once our primary theorem is stated.

Parallel to our study of (\ref{schro}), we will consider (possibly sign changing) weak solutions of
\begin{equation}\label{ric}
-\text{div}(|\nabla v|^{p-2}\nabla v) = (p-1)|\nabla v|^p + \sigma  \text{ in } \Omega.
\end{equation}
This equation is of interest in its own right in nonlinear PDE, for instance see the paper of Ferone and Murat \cite{FM00}, and references therein.  Related problems are considered in \cite{AHBV09, Por02, MP02, GT03, ADP06, PS06}.  The critical $p$-growth in the gradient term appearing on the right hand side of (\ref{ric}) means a strong a priori bound is required to overcome weak convergence issues and prove the existence of solutions to (\ref{ric}).  In this paper, we employ a well-known connection between solutions of (\ref{ric}) and (\ref{schro}) with the aid of the substitution $v=\log(u)$, where $u$ a positive solution of (\ref{schro}).  This substitution is known to be delicate, especially when going from the equation (\ref{ric}) to the equation (\ref{schro}), see \cite{FM00}.  There are several recent works devoted to questions related to this substitution, see for example \cite{AHBV09, KKT11} and references therein.

Our second result, Theorem \ref{sobcor}, illustrates the utility of our work on the equations (\ref{schro}) and (\ref{ric}).  It regards a characterization of the Sobolev-Poincar\'{e} inequality (\ref{pformbd}).   This result is a direct extension of the $p=2$ case already studied in \cite{MV1}.  Our characterization of this inequality for $p\not= 2$, which is of substantial interest, comes as a relatively straightforward corollary of our main results for the elliptic equations.

Since it is the effect of the lower order term $\sigma|u|^{p-2}u$ on the differential operator which is of most interest here, we have introduced equations (\ref{schro}) and (\ref{ric}) with the $p$-Laplacian operator.  However, our methods extend to quasilinear operators with more general structure discussed in Section \ref{general}.

It is not obvious how one makes sense of solutions to equation (\ref{schro}) under the sole condition (\ref{pformbd}), while for  equation (\ref{ric}) the situation is more straightforward.
We make the following definition:

\begin{defn}[Weak solutions]\label{schrodef} Let $\sigma \in L^{-1, p'}_{\rm{loc}} (\Omega)$. 

(i).  A function $u$ is a \textit{weak solution} of (\ref{schro}) if both $u\in L^{1,p}_{\rm{loc}}(\Omega)$ and $|u|^{p-2}u\in L^{1,p}_{\rm{loc}}(\Omega)$, and (\ref{schro}) holds in the sense of distributions.  In other words, for any function $\varphi\in C^{\infty}_0(\Omega)$ one has
$$\int_{\Omega}|\nabla u|^{p-2}\nabla u\cdot \nabla \varphi \,dx = \langle \sigma, |u|^{p-2}u\varphi\rangle.
$$

(ii).  A function $v$ is a \textit{weak solution} of (\ref{ric}) if  $v\in L^{1,p}_{\rm{loc}}(\Omega)$, and (\ref{ric}) holds in the sense of distributions.  This means that for any $\varphi\in C^{\infty}_0(\Omega)$, one has
$$\int_{\Omega}|\nabla v|^{p-2}\nabla v\cdot \nabla \varphi \,dx = (p-1)\int_{\Omega}|\nabla v|^p \varphi dx + \langle \sigma, \varphi\rangle.
$$
\end{defn}

Using Definition \ref{schrodef}, all terms in (\ref{schro}) and (\ref{ric})  are well defined as distributions.  A function $u\in L^{1,p}_{\text{loc}}(\Omega)$ will be called positive if there exists $E\subset \Omega$ with $u(x)>0$ for all $x\in \Omega\backslash E$ and $\text{cap}_p(E,\Omega)=0$  (see (\ref{cap})).  This is not an artificial definition, as there are simple examples of $\sigma$ which should be included in our study, for which all positive weak solutions have an interior zero.

The two inequalities contained in (\ref{pformbd}) are responsible for different properties of solutions to (\ref{schro}).  We therefore separate them into an upper bound
\begin{equation}\label{1pformupper}
\langle |h|^p, \sigma \rangle  \leq \lambda \int_{\Omega} |\nabla h|^p dx, \text{ for all } h\in C^{\infty}_0(\Omega),
\end{equation}
and a lower bound
\begin{equation}\label{1pformlower}
 - \langle |h|^p, \sigma \rangle  \leq \Lambda \int_{\Omega}|\nabla h|^p dx, \text{ for all } h\in C^{\infty}_0(\Omega).
\end{equation}
There will be no smallness condition required on the parameter $\Lambda>0$ in (\ref{1pformlower}) at any point in the paper.

We are now in a position to state our main theorem.  Let $p^{\#}$ be defined by $p^{\#} = (p-1)^{2-p}$ if $p\geq 2$, and $p^{\#} = 1$ if $1<p\leq 2$.

\begin{thm}\label{introthmequ}  
The following statements hold:

(i) Suppose that $\sigma \in L^{-1,p'}_{\rm{loc}}(\Omega)$ satisfies $(\ref{1pformupper}) \text{ with }\lambda\in (0,p^{\#})$, and $(\ref{1pformlower})$ with $\Lambda>0.$  Then there exists a positive weak solution $u$ of (\ref{schro}) (see Definition \ref{schrodef}) satisfying
\begin{equation}\label{1mult1}
\int_{\Omega} \frac{|\nabla u|^p}{u^p}|h|^p dx\leq C_0\int_{\Omega}|\nabla h|^p dx,\; \text{ for all }h\in C^{\infty}_0(\Omega),
\end{equation}
for a constant $C_0=C_0(\Lambda,p)$.
Furthermore, if $v=\log(u)$, then $v\in L^{1,p}_{\rm{loc}}(\Omega)$ is a weak solution of (\ref{ric}) satisfying
\begin{equation}\label{1mult2}
\int_{\Omega} |\nabla v|^p|h|^p dx\leq C_0\int_{\Omega}|\nabla h|^p dx, \;\text{ for all }h\in C^{\infty}_0(\Omega).
\end{equation}
\indent (ii)  Conversely, if there is a solution $v\in L^{1,p}_{\rm{loc}}(\Omega)$ of (\ref{ric}) so that (\ref{1mult2}) holds for a constant $C_0$, then the inequality (\ref{1pformupper}) holds with $\lambda=1$ and (\ref{1pformlower}) holds for a constant $\Lambda = \Lambda(C_0)$.
\end{thm}

In part (i) of Theorem \ref{introthmequ}, the local regularity of the solution to (\ref{schro}) is optimal (i.e. cannot be replaced by $L^{1,q}_{\text{loc}}(\Omega)$ for any $q>p$).  This is the case even when $p=2$, see \cite{JMV11}.

\begin{rem}\label{nonlinregrem}  The condition $0<\lambda<p^{\#}$ is sharp in order to obtain a solution of (\ref{schro}) in the sense of Definition \ref{schrodef}.  If $\Omega = \mathbf{R}^n$, this can be seen from working with the potential
$$\sigma  = t\cdot c_0 |x|^{-p}, \text{ for } c_0= \Bigl(\frac{n-p}{p}\Bigl)^p, \text{ and }0<t\leq1.$$
If $t=1$, then (\ref{1pformupper}) holds with best constant $\lambda = 1$ by the classical multidimensional variant of Hardy's inequality.  An elementary calculation shows that the equation (\ref{schro}) has a positive solution $u(x) = |x|^{\gamma}$, with $\gamma = \gamma(t,n,p)$, for all $t\in (0,1]$.

If $p\geq 2$ and $t = (p-1)^{2-p}$, we may choose
$\gamma  =  \frac{p-n}{p(p-1)}$,  in which case, the solution $u(x) = |x|^{\gamma}$ is the unique (up to constant multiple) positive solution of (\ref{schro}) in $L^{1,p}_{\text{loc}}(\mathbf{R}^n)$  (see \cite{Pol03,PS05}).  Notice that $u^{p-1}\not\in L^{1,p}_{\text{loc}}(\mathbf{R}^d)$, and hence $u$ is not a solution in the sense of Definition \ref{schrodef}.   For all $p>1$ and $t=1$, we have $ \gamma  = \frac{p-n}{p},$  and the resulting solution $u$ is the unique (up to constant multiple) positive distributional solution of (\ref{schro}) (see \cite{PS05}).  Note that $|x|^{(p-n)/p}$ does not lie in $L^{1,p}_{\text{loc}}(\Omega)$, and therefore the assumption that $\lambda<1$ in Theorem \ref{introthmequ} cannot be relaxed.

Notice that the above example also shows that one cannot expect global integrability properties of solutions of (\ref{schro}) (at least in unweighted Sobolev spaces). 
\end{rem}

The solution of (\ref{schro}) obtained in Theorem \ref{introthmequ} may possess improved integrability properties if one has better control of the parameter $\lambda>0$.   This follows from a slight modification of the method of Br\'{e}zis and Kato \cite{BK79}, and is carried out in Section \ref{highintrem}.

The heart of the proof of Theorem \ref{introthmequ} lies in proving part (i).  Here the proof breaks off into two parts.  The first part consists of establishing local $L^p$-estimates on the gradient of a suitable approximating sequence.  This follows a similar path to the linear case $p=2$ previously presented in \cite{JMV11}, where doubling properties are used in order to compensate for a lack of compactness.  The second part of the proof concerns the passage to the limit, where there are significant hurdles.  We follow the general scheme spelled out in the important papers \cite{BBGVP, DMMOP} in reducing matters to certain convergence in measure properties.  However, the proof of the required convergence in measure will be quite non-trivial on the basis of the distributional nature of $\sigma$, and several judicious choices of test functions will be required.



\subsection{A characterization of the inequality (\ref{pformbd})}We shall now state our characterization of the Sobolev-Poincar\'{e} inequality (\ref{pformbd}).  
We focus on the case when $\Omega = \mathbf{R}^n$.  From this case, one can deduce a similar characterization of (\ref{pformbd}) for bounded domains $\Omega$ which support a Hardy inequality of the following form: There exists $C>0$ such that
$$\int_{\Omega}\frac{|h(x)|^p}{\text{dist}(x, \partial\Omega)^p} dx\leq C\int_{\Omega}|\nabla h|^p dx, \text{ for all }h\in C^{\infty}_0(\Omega).
$$
This reduction is spelled out in \cite{MV1}.  Furthermore, we will consider only $n\geq 2$, since the one dimensional case was previously studied in \cite{MV2}.  

\begin{thm} \label{sobcor} Let $n\geq 2$, and suppose $\sigma \in L^{-1,p'}_{\rm{loc}}(\Omega)$.  For a constant $C_0>0$, the inequality
\begin{equation}\label{pformbdentire}
|\langle |h|^p, \sigma \rangle | \leq C_0 \int_{\mathbf{R}^n} |\nabla h|^p dx, \text{ for all } h\in C^{\infty}_0(\mathbf{R}^n),
\end{equation}
holds if and only if

(i)  $1<p<n$, and there exists $C_1 = C_1(C_0, n, p)$ and $\vec\Gamma\in (L^{p'}_{\rm{loc}}(\mathbf{R}^n))^n$ such that one can represent $\sigma = {\rm div}(\vec\Gamma)$, with $\vec\Gamma$ satisfying
\begin{equation}\label{positiveineq}
\int_{\mathbf{R}^n}|h|^p |\vec\Gamma|^{p'}dx \leq C_1\int_{\mathbf{R}^n} |\nabla h|^p dx \text{ for all }h\in C^{\infty}_0(\mathbf{R}^n).
\end{equation}
\indent (ii)   $p\geq n$,  and $\sigma \equiv0$.
\end{thm}


The strength of Theorem \ref{pformbdentire} lies in recasting the inequality (\ref{pformbdentire}) with indefinite weight $\sigma$, in terms of the inequality (\ref{positiveineq}) with positive weight $|\vec \Gamma|^{p'}$.  The inequality (\ref{positiveineq}) has a rich history in its own right, and can be recast in terms of a capacity condition (see \cite{Ad09},  Chapter 2 of \cite{Maz85}, or Chapter 7 of \cite{AH96}).  Combining this result with Theorem \ref{sobcor}, we arrive at the following corollary.

\begin{cor}Let $p\in (1,n)$.  Then (\ref{pformbdentire}) holds if and only if there exists $C_1=C_1(C_0,n,p)$ and $\vec{\Gamma}\in (L^{p'}_{\rm{loc}}(\mathbf{R}^n))^n$, such that
$\sigma  = {\rm{div}} ( \vec{\Gamma} )$, with $\vec{\Gamma}$ satisfying
\begin{equation}\label{characcapcond}
\int_{E} |\vec\Gamma|^{p'} dx \leq C\rm{cap}_p(E) \text{ for all compact sets }E\subset\mathbf{R}^n.
\end{equation}
\end{cor}Here $\text{cap}_p(E) = \text{cap}_p(E,\mathbf{R}^n)$ is the capacity associated with the homogeneous Sobolev space $L^{1,p}(\mathbf{R}^n)$.  For a general open set $\Omega\subset \mathbf{R}^n$, and a compact set $E\subset \Omega$, we define
\begin{equation}\label{cap}
\text{cap}_p(E, \Omega) = \inf\{||\nabla h||_{L^p(\Omega)}^p \, :\, h\geq 1 \text{ on }E, \, \, h\in C^{\infty}_0(\Omega)\}.
\end{equation}

   Several  conditions equivalent  to (\ref{characcapcond}) (or (\ref{positiveineq})) which do not involve capacities
 are known (see for example \cite{Maz85}, Sec. 1.2.5, and \cite{V}).

\subsection{The plan of the paper}The plan of the paper is as follows.  In Section \ref{general} we formulate our main theorem in the framework of a more general quasilinear operator.  Then in Section \ref{prelim} we develop the required preliminaries.  Section \ref{equthmsec} is the heart of the paper, and Theorem \ref{introthmequ} is proved there.  In Section \ref{highintrem}, we remark on additional integrability properties for solutions of (\ref{schro}).  Finally, Section \ref{sobthmsec} is devoted to deducing Theorem \ref{sobcor} from Theorem \ref{introthmequ}.

\section{The main result for the general operator}\label{general}

Since our techniques do not use the particular structure of the $p$-Laplacian operator, we state a version of Theorem \ref{introthmequ} for more general operators.  In a couple of places in our argument, we will sacrifice generality for ease of exposition, but in such instances we will indicate how the argument can be extended.

For an open set $\Omega\subset \mathbf{R}^n$, let $\mathcal{A}:\Omega \times \mathbf{R}^n \rightarrow \mathbf{R}$ be measurable in the first variable for each $\xi\in \mathbf{R}^n$, and continuous in the second variable for almost every $x\in \Omega$.  In addition, suppose $\mathcal{A}$ satisfies the following conditions:
\begin{enumerate}
\item (Ellipticity and boundedness) There exists $0<m\leq M$ so that for almost every $x\in \Omega$ and for all $\xi \in \mathbf{R}^n$,
\begin{equation}\label{ellip}
\mathcal{A}(x, \xi)\cdot \xi \geq m|\xi|^p, \text{ and } \, \, |\mathcal{A}(x, \xi)| \leq M|\xi|^{p-1}.
\end{equation}
\item (Homogeneity) For all $\xi\in\mathbf{R}^n$, and almost every $x\in \Omega$,
\begin{equation}\label{homog}
\mathcal{A}(x,t\xi) = |t|^{p-2}t\mathcal{A}(x,\xi), \text{ for any }t\in \mathbf{R}.
\end{equation}
\item (Monotonicity)  There exists a constant $c>0$ such that, for almost every $x\in \Omega$, and for all $\xi$, $\eta \in \mathbf{R}^n\setminus\{0\}$,
\begin{equation}\label{monotone}
(\mathcal{A}(x, \xi) - \mathcal{A}(x, \eta))\cdot (\xi -\eta)\geq c\begin{cases}
|\xi - \eta|^p, \text{ if }p\geq 2,\\
\displaystyle \frac{|\xi-\eta|^2}{(|\xi|^{2-p}+ |\eta|^{2-p})} \text{ if }1<p\leq 2.
\end{cases}
\end{equation}
\item (Continuity) There exists a modulus of continuity $\omega : [0,\infty) \rightarrow [0,\infty)$ so that $\lim_{\varepsilon \rightarrow 0^+} \omega(\varepsilon) = 0$, and
\begin{equation}\label{cont}
|\mathcal{A}(x,\xi) - \mathcal{A}(y,\xi)| \leq \omega(|x-y|) |\xi|^{p-1}.
\end{equation}
\item (Convexity) For almost every $x\in \Omega$, the function
\begin{equation}\label{conv}
\xi \rightarrow \mathcal{A}(x,\xi)\cdot \xi \text{ is convex on }\mathbf{R}^n.
\end{equation}
\end{enumerate}
These assumptions will be in force for the remainder of this paper, unless stated otherwise.  

\begin{rem} The convexity assumption here is natural for our problem.  Indeed, in our more general version of the Sobolev-Poincar\'{e} inequality (\ref{pformbd}) (see (\ref{gensob}) below), the convexity condition guarantees the right hand side (when raised to the power $1/p$) is a semi-norm.  In the linear case $p=2$, and $\mathcal{A}(\cdot, \xi) = a_{ij}(\cdot)\xi_i$, it is routine to check that convexity is a consequence of ellipticity of the matrix $(a_{ij})$.  One can obtain existence results without the convexity assumption if one permits a smaller constant in the inequality (\ref{pformupper}) below (see Remark \ref{noconv}).

It seems likely that the continuity assumption on the operator (condition (\ref{cont})) can be weakened.  This assumption is used to obtain a certain convergence of measure result (carried out in Section \ref{convmeassec}).  Since this convergence result is quite technical, we decided not to complicate matters by introducing a more refined regularity assumption on the operator.  In the linear case $p=2$, there is no need for any continuity assumption, see \cite{JMV11}.
\end{rem}

With the conditions on our operator stated, we are now in a position to state our main result.  We will consider solutions of the following equations, which are the natural generalizations of (\ref{schro}) and (\ref{ric}) respectively:
\begin{equation}\label{genschro}
-\text{div}(\mathcal{A}(x,\nabla u)) = \sigma |u|^{p-2}u \text{ in }\Omega,
\end{equation}
and,
\begin{equation}\label{genric}
-\text{div}(\mathcal{A}(x,\nabla v)) = (p-1)\mathcal{A}(x,\nabla v)\cdot\nabla v + \sigma \text{ in }\Omega.
\end{equation}
We will consider solutions of (\ref{genschro}) as in Definition \ref{schrodef}, with (\ref{schro}) replaced by (\ref{genschro}).  The more general variant of the Sobolev-Poincar\'{e} inequality (\ref{pformbd}) in this context is
\begin{equation}\label{gensob}|\langle \sigma, |h|^p \rangle| \leq C\int_{\Omega} \mathcal{A}(x,\nabla h)\cdot \nabla h \,dx, \text{ for all }h\in C^{\infty}_0(\Omega).
\end{equation}
\begin{thm}\label{mainthm}  Suppose $\Omega$ is an open set, and suppose that $\mathcal{A}$ satisfies the assumptions (\ref{ellip})--(\ref{conv}).

(i) Suppose that $\sigma \in L^{-1,p'}_{\rm{loc}}(\Omega)$ satisfies
\begin{equation}\label{pformupper}
\langle |h|^p, \sigma \rangle  \leq \lambda \int_{\Omega} \mathcal{A}(x,\nabla h)\cdot \nabla h dx, \text{ for all } h\in C^{\infty}_0(\Omega),
\end{equation}
with $\lambda\in (0, p^{\#})$.  In addition, suppose that
\begin{equation}\label{pformlower}
 - \langle |h|^p, \sigma \rangle  \leq \Lambda \int_{\Omega}\mathcal{A}(x,\nabla h)\cdot \nabla h dx, \text{ for all } h\in C^{\infty}_0(\Omega),
\end{equation}
for some $\Lambda>0$.  Then there exists a positive weak solution $u$ of (\ref{genschro}) (see Definition \ref{schrodef}) satisfying
\begin{equation}\label{mult1}
\int_{\Omega} \frac{|\nabla u|^p}{u^p}|h|^p dx\leq C_0\int_{\Omega}|\nabla h|^p dx,\; \text{ for all }h\in C^{\infty}_0(\Omega),
\end{equation}
for a positive constant $C_0=C_0(\Lambda, p)$.
Moreover, if $v=\log(u)$, then $v\in L^{1,p}_{\rm{loc}}(\Omega)$ is a weak solution of (\ref{genric}) satisfying
\begin{equation}\label{mult2}
\int_{\Omega} |\nabla v|^p|h|^p dx\leq C_0\int_{\Omega}|\nabla h|^p dx, \;\text{ for all }h\in C^{\infty}_0(\Omega).
\end{equation}
\indent (ii)  Conversely, if there is a solution $v\in L^{1,p}_{\rm{loc}}(\Omega)$ of (\ref{genric}) so that (\ref{mult2}) holds for a constant $C_0$, then the inequality (\ref{pformupper}) holds with $\lambda= (M/m)^p$ and (\ref{pformlower}) holds for a constant $\Lambda = \Lambda(C_0, M)$.
\end{thm}

One can also carry out a more local version of Theorem \ref{mainthm}, akin to Theorem 1.2 in \cite{JMV11}, using only local conditions on the operator $\mathcal{A}$ and potential $\sigma$.  Since all our arguments are local, this is a straightforward modification of the proof that follows, cf. Section 3 of \cite{JMV11}.

\begin{rem}\label{noconv}  With convexity assumption on $\mathcal{A}$ removed, one can still reach the conclusion of the part (i) of Theorem \ref{mainthm}, provided that
$\lambda < \frac{m}{M} p^{\#}$.  We leave it to the reader to check that Lemmas \ref{est1} and \ref{est2} below can be pushed through in this range of $\lambda$ without the convexity property.  The two lemmas just mentioned are where the convexity plays a role.
\end{rem}

Regarding statement (ii) of Theorem \ref{mainthm}, it is the case even when $p=2$ that in general the constant $\lambda$ needs to depend on $M$ and $m$.  See Section 7 of \cite{JMV11}.

\section{Preliminaries}\label{prelim}

We re-iterate that throughout this paper we will assume (unless stated otherwise) that $\mathcal{A}:\Omega\times \mathbf{R}^n\rightarrow \mathbf{R}^n$ satisfies (\ref{ellip})--(\ref{conv}).

\subsection{Notation}\label{notation}  We shall denote by $C$ a positive constant which may depend on $n$, $p$, $m$, and $M$.  Any additional dependencies (beyond $n$, $p$, $m$, and $M$) of a constant $C$ will be stated explicitly, for example a constant $C(\lambda)$ may depend on $\lambda$, as well as $n$, $p$, $m$, and $M$.  Within a proof a constant $C$ may change from line to line.

For an open set $\Omega\subset \mathbf{R}^n$, we denote by $C^{\infty}_0(\Omega)$ the space of infinitely differentiable functions with compact support in $\Omega$.  Define the energy space $L^{1,p}_0(\Omega)$ to be the completion of $C^{\infty}_0(\Omega)$ in the semi-norm $||\nabla (\, \cdot\,)||_{L^p(\Omega)}$.  We say that $f\in L^{1,p}_{\text{loc}}(\Omega)$ if $f\in L^{p}_{\text{loc}}(\Omega)$ and $f\varphi \in L^{1,p}_0(\Omega)$  for all $\varphi\in C^{\infty}_0(\Omega)$.  Let $L^{-1,p'}(\Omega)$ be the dual space of $L^{1,p}_0(\Omega)$.  We say that $\sigma \in L^{-1,p'}_{\text{loc}}(\Omega)$ if $\sigma \varphi \in L^{-1,p'}(\Omega)$ whenever $\varphi\in C^{\infty}_0(\Omega)$.  We define $W^{1,p}(\Omega)$ as those functions $f\in L^p(\Omega)$ with weak derivative $\nabla f\in (L^{p}(\Omega))^n$.

If $\Omega\subset \mathbf{R}^n$ is a \emph{bounded} open set, then the Poincar\'{e} inequality guarantees that $L^{1,p}_0(\Omega)$ is a Banach space with norm $||\nabla (\, \cdot\,)||_{L^p(\Omega)}$, for any $1<p<\infty$, see for example Br\'{e}zis \cite{Bre11}, Corollary 9.19.  We shall only use that $L^{1,p}_0(\Omega)$ is a Banach space when the underlying set $\Omega$ is bounded\footnote{If $p<n$, then the result is true regardless of $\Omega$, but it is a delicate issue for unbounded sets if $p\geq n$, see \cite{Maz85}, Sec. 15.2}.  Provided $\Omega$ is a bounded set, every $\sigma \in L^{-1,p'}(\Omega)$ may be represented as $\sigma  = \text{div}(\vec T)$, with $\vec T\in (L^{p'}(\Omega))^n$, see for example \cite{Bre11},  Proposition 9.20.

As a result of the previous discussion,  if $\sigma\in L^{-1,p'}_{\text{loc}}(\Omega)$, and $U$ is a bounded open subset of $\Omega$, then there exists with $\vec T\in (L^{p'}(\Omega))^n$ such that $\sigma = \text{div}(\vec T)$ in $U$.  Therefore, if $u\in L_0^{1,p}(U)$, then
$$\langle u, \sigma \rangle = \int_{U} \nabla u\cdot \vec T dx. 
$$

\subsection{Local smoothing}

We begin this section with some remarks about mollification.  Fix $\varphi$ so that $\varphi\in C^{\infty}_0(B_1(0))$, $\varphi\geq 0$, $\varphi$ is radially symmetric, and $\int_{B_1(0)}\varphi(x)dx =1.$
We denote $\varphi_{\varepsilon}(x) = \varepsilon^{-n} \varphi(x/\varepsilon).$

For the majority of this paper, we will use the mollified operator $\mathcal{A}_{\varepsilon}$, defined (for a smooth function $f$) by
\begin{equation}\label{Aepsilon}\mathcal{A}_{\varepsilon}(x, \nabla f(x)) = \int_{B(x,\varepsilon)} \varphi_{\varepsilon}(y)\mathcal{A}(x+y, \nabla f(x))dy.
\end{equation}
In other words, we only mollify the spatial variable, and leave the gradient variable unchanged.
\begin{rem}  Let $\varepsilon>0$ and suppose $U\subset\subset \Omega$ with ${\rm{dist}}(U, \partial\Omega)>\varepsilon$.  Then $\mathcal{A}_{\varepsilon}: U\times \mathbf{R}^n\rightarrow \mathbf{R}^n$ satisfies (\ref{ellip})--(\ref{monotone}) and (\ref{conv}) inside $U$.
\end{rem}

The next lemma concerns how the inequality (\ref{gensob}) behaves under mollification:

\begin{lem}\label{mollem}  Let $U\subset\subset \Omega$, and let $\varepsilon>0$ be such that $\varepsilon<{\rm{dist}}(U, \partial\Omega)$.  Let $\lambda>0$, suppose that $\sigma$ satisfies
\begin{equation}
\langle |h|^p ,\sigma \rangle  \leq \lambda \int_{\Omega} \mathcal{A}(\cdot,\nabla h)\cdot\nabla h dx, \text{ for all }h\in C^{\infty}_0(\Omega).
\end{equation}
Then, with $\sigma_{\varepsilon} = \varphi_{\varepsilon}* \sigma$, we have
\begin{equation}\label{smoothemb}
\int_U |h|^p d\sigma_{\varepsilon} \leq \lambda \int_U \mathcal{A}_{\varepsilon}(\cdot,\nabla h)\cdot \nabla h dx, \text{ for all }h\in L^{1,p}_0(U),
\end{equation}
where $d\sigma_{\varepsilon} = \sigma_{\varepsilon} dx$.
\end{lem}

\begin{proof}
Notice that $\sigma_{\varepsilon}\in C^{\infty}(\overline{U})$.  By density, and the continuity of $\sigma_{\varepsilon}$ and $\mathcal{A}_{\varepsilon}$, it suffices to prove (\ref{smoothemb}) for $h\in C^{\infty}_0(U)$.  We first note that by the interchange of mollification and the distribution (see for example Lemma 6.8 of \cite{LL01}), we have
$$\langle \sigma, \varphi_{\varepsilon}*|h|^p\rangle= \int_{B(0,\varepsilon)}\varphi_{\varepsilon}(t) \langle\sigma, |h(\cdot-t)|^p \rangle dt.
$$
By choice of $U$ and $\varepsilon$, note that $h(\,\cdot-t)\in C^{\infty}_0(\Omega)$ for all $t\in B_{\varepsilon}(0)$.  Hence
\begin{equation}\begin{split}\nonumber\langle\sigma_{\varepsilon}, |h|^p\rangle & \leq \lambda \int_{B(0,\varepsilon)} \varphi_{\varepsilon}(t)\Bigl(\int_{U}\mathcal{A}(x,\nabla h(x-t))\cdot \nabla h(x-t) dx\Bigl)dt \\
&= \lambda \int_{\Omega}\mathcal{A}_{\varepsilon}(x, \nabla h(x))\cdot \nabla h(x) dx,
\end{split}\end{equation}
which proves the lemma.
\end{proof}

The convexity property (\ref{conv}) combines with the homogeneity property (\ref{homog}) to yield Minkowski's inequality:   For $\Gamma_1, \Gamma_2\in (L^p(\Omega))^n$, we have
\begin{equation}\label{mink}\begin{split}
\Bigl(\int_{\Omega} \mathcal{A}&(\cdot, \Gamma_1+\Gamma_2) \cdot(\Gamma_1 + \Gamma_2)dx\Bigl)^{1/p} \\
&\leq \Bigl(\int_{\Omega} \mathcal{A}(\cdot, \Gamma_1)\cdot(\Gamma_1)dx\Bigl)^{1/p} + \Bigl(\int_{\Omega} \mathcal{A}(\cdot, \Gamma_2)\cdot( \Gamma_2)dx\Bigl)^{1/p}.
\end{split}\end{equation}
The same statement holds for $\mathcal{A}_{\varepsilon}$ inside $U\subset\subset\Omega$, provided ${\rm{dist}}(U, \partial\Omega)>\varepsilon$: for any $\Gamma_1, \Gamma_2\in (L^p(U))^n$,
\begin{equation}\nonumber\begin{split}
\Bigl(\int_{U} \mathcal{A}_{\varepsilon}&(\cdot, \Gamma_1+\Gamma_2) \cdot(\Gamma_1 + \Gamma_2)dx\Bigl)^{1/p} \\
&\leq \Bigl(\int_{U} \mathcal{A}_{\varepsilon}(\cdot, \Gamma_1)\cdot(\Gamma_1)dx\Bigl)^{1/p} + \Bigl(\int_{U} \mathcal{A}_{\varepsilon}(\cdot, \Gamma_2)\cdot( \Gamma_2)dx\Bigl)^{1/p}.
\end{split}\end{equation}


\subsection{Local existence}

The next lemma concerns a local existence result. It will be used to produce a sequence of approximate solutions to (\ref{genschro}).

\begin{lem}\label{existencelem}  Suppose that $V$ is a bounded domain with a ball $B\subset\subset V$.  Suppose that $\tilde{\mathcal{A}} : V\times \mathbf{R}^n \rightarrow \mathbf{R}$ satisfies (\ref{ellip})--(\ref{monotone}).  For $0<\lambda<1$, let $\tilde \sigma \in C^{\infty}(\overline{V})$ satisfy
\begin{equation}\label{tildesigmaest}\int_{V}|h|^p d\tilde\sigma \leq \lambda \int_{V}\tilde{\mathcal{A}}(x, \nabla h)\cdot \nabla h dx, \text{ for all }h\in C^{\infty}_0(V).
\end{equation}
Then there exists a positive solution $v\in C^{\alpha}_{\text{loc}}(V) \cap W^{1,p}(V)$ of $$-{\rm div}(\tilde{\mathcal{A}}(\cdot,\nabla v)) = \tilde \sigma v^{p-1},$$
so that $\int_B v^{qp} dx = 1$.  Here $q=\max(p-1,1)$.
Furthermore, $v$ satisfies the Harnack inequality in $V$.
\end{lem}

\begin{proof}  The existence of a solution follows from the  theory of monotone operators.  Indeed, note that (\ref{tildesigmaest}) guarantees coercivity in the Sobolev space $L^{1,p}_0(V)$ of the operator $\mathcal{A}(\cdot, \nabla v) - \tilde\sigma |v|^{p-2}v$.  On the other hand, the smoothness of $\tilde\sigma$ ensures that the associated functional is weakly continuous.  From these two facts one can follow standard theory, see e.g. Chapter 6 of \cite{MZ97}, to obtain the existence of a solution of the equation
\begin{equation}\nonumber
\begin{cases}
-\text{div}(\tilde{\mathcal{A}}(\cdot, \nabla \tilde v)) = \tilde \sigma |\tilde v|^{p-2}\tilde v,\\
\tilde v -1 \in L^{1,p}_0(V).
\end{cases}\end{equation}
The solution $\tilde v$ is nonnegative.  To see this, test the weak formulation of the preceding equation with $h = \min(0, \tilde v )\in L^{1,p}_0(V)$.  Note that
$$\int_{V}\tilde{\mathcal{A}}(\cdot, \nabla h)\cdot \nabla h dx = \int_{V}\tilde{\mathcal{A}}(\cdot, \nabla \tilde v)\cdot \nabla h dx = \int_V |\tilde v |^{p-2}\tilde v h \tilde \sigma dx.
$$
The last integral on the right hand side is equal to $\int_V |h|^p \tilde\sigma dx$. Applying (\ref{tildesigmaest}), this integral is in turn less than $\lambda\int_{V}\tilde{\mathcal{A}}(\cdot, \nabla h)\cdot \nabla h dx$, and hence $(1-\lambda)\int_{V}\tilde{\mathcal{A}}(\cdot, \nabla h)\cdot \nabla h dx\leq 0$. Since $\lambda\in (0,1)$, it follows from (\ref{ellip}) that $\int_{V}|\nabla h|^p dx = 0$.  We conclude that $\min(0, \tilde v) = 0$ quasi-everywhere, as required.

Using the smoothness of $\tilde\sigma$, we apply the results of Serrin \cite{Ser2} to yield the Harnack inequality for $\tilde v$, along with the H\"{o}lder continuity (for all $1<p<\infty$).  To conclude the proof, it remains to renormalize $\tilde v$ in order to obtain the given integrability property on $B$.
\end{proof}

\subsection{Weak reverse H\"{o}lder inequalities and BMO}\label{wrhsec}

In this section, we recall a result from \cite{JMV11} regarding weak reverse H\"{o}lder inequalities.  For $p\in (1, \infty)$ and an open set $U$, we say $u\in BMO(U)$ if there is a positive constant $D_U$ such that
\begin{equation}\label{BMOdef}\dashint_{B(x,r)} \!|u(y)-\dashint_{B(x,r)} \!u(z) dz|^p dy \leq D_U, \!\text{ for any ball }B(x,2r)\!\subset\! U.
\end{equation}
A well known consequence of the John-Nirenberg inequality is that one can replace the exponent $p$ in (\ref{BMOdef}) with any $0<q<\infty$, and obtain a comparable semi-norm.  We say that $u\in BMO_{\text{loc}}(\Omega)$ if for each compactly supported open set $U\subset\subset\Omega$, there is a positive constant $D_U>0$ so that (\ref{BMOdef}) holds.

\begin{defn}  Let $U\subset\mathbf{R}^n$ be an open set.

(a) A nonnegative measurable function $w$ is said to be \textit{doubling} in $U$ if there exists a constant $A_U>0$ such that
\begin{equation}\label{ld}\dashint_{B(x,2r)} w\, dx \leq A_U\dashint_{B(x,r)} w \, dx, \, \, \textrm{for all balls }B(x,4r)\subset U.
\end{equation}

(b) A nonnegative measurable function $w$ is said to satisfy a \textit{weak reverse H\"{o}lder inequality} in $U$ if there exist constants $q>1$ and $B_U>0$ such that
\begin{equation}\label{wrh}\Bigl(\dashint_{B(x,r)} w^q dx\Bigl)^{1/q} \leq B_U\dashint_{B(x,2r)} w\,dx, \, \, \textrm{for all balls }B(x,2r)\subset U.
\end{equation}
\end{defn}

Our argument hinges on the following result.

\begin{prop} \label{locdoub} Let $U$ be an open set. Suppose that $w$ satisfies the weak reverse H\"{o}lder inequality (\ref{wrh}) in $U$.  Then $w$ is doubling in $U$ if and only if $\log(w) \in BMO(U)$.  In particular, suppose $w$ satisfies (\ref{wrh}), and there exists a constant $D_U$ such that for any ball $B(x, 2s)\subset U$
\begin{equation}\label{bmo1}
\dashint_{B(x,s)} \!|\! \log w(y) - \dashint_{B(x,s)} \!\log w(z) dz|^p dy \leq D_U.
\end{equation}
Then there is a constant $C(B_U, D_U)>0$, such that for any ball $B(x,4r)\subset U$,
\begin{equation}\label{doub}
\dashint_{B(x,2r)} w\, dx \leq C(B_U, D_U) \dashint_{B(x,r)} w\, dx.
\end{equation}
\end{prop}

For the proof of this proposition, see \cite{JMV11}, Proposition 2.3.

\section{Proof of the main result}\label{equthmsec}

\subsection{Proof of part (i) of Theorem \ref{mainthm}}

 Suppose that the hypotheses of part (i) from Theorem \ref{mainthm} are in force.  We shall assume henceforth that $\Omega$ is a connected open set.  This assumption is made without loss of generality since we can apply the argument below in each connected component.  The assumption of connectedness is used in a Harnack chain argument.  We begin by constructing an approximating sequence.

\subsection{Construction of an approximating sequence}  Let $(\Omega_j)_j$ be an exhaustion of $\Omega$ by smooth connected domains, in other words $\Omega_j\subset\subset\Omega_{j+1}$ and $\cup_j \Omega_j = \Omega$,  see for example \cite{EE87}.  In addition, fix a ball $B$  so that $8B\subset \Omega_1$.

Let $\varepsilon_j = \min(2^{-j}, \tfrac{1}{2}d(\Omega_j, \partial\Omega_{j+1}))$, and put $\sigma_j = \varphi_{\varepsilon_j}*\sigma$ and $\mathcal{A}_j = \mathcal{A}_{\varepsilon_j}$, with $\varphi_{\varepsilon_j}$ and $\mathcal{A}_{\varepsilon_j}$ as in Lemma \ref{mollem}.  Note that $\mathcal{A}_j$ satisfies (\ref{ellip})--(\ref{monotone}) and (\ref{conv}) in $\Omega_j$.

Applying Lemma \ref{mollem}, it follows that (\ref{tildesigmaest}) holds with $\tilde\sigma = \sigma_{\varepsilon_j}$, $\tilde{\mathcal{A}} = \mathcal{A}_j$ and $V=\Omega_j$. As a result, the hypotheses of Lemma \ref{existencelem} are fulfilled, and we deduce the existence of a sequence $(u_j)_j$  of positive functions satisfying
\begin{equation}\label{approxsol}
\begin{cases}-\text{div}(\mathcal{A}_j(\cdot, \nabla u_j)) = \sigma_j u^{p-1}_j \text{ in }\Omega,\\
\displaystyle \int_{B} u_j^{qp} dx = 1.
\end{cases}\end{equation}
Here $q = \max(p-1,1)$, as before.  In addition, within each $\Omega_j$ the function $u_j$ satisfies the Harnack inequality (of course the implicit constants in these estimates blow up with $j$ and will be only used qualitatively).  Our first task will be to prove a  local gradient estimate for the tail of the sequence $(u_k)_{k>j}$ inside $\Omega_j$.
\begin{prop}\label{locuniform} For a fixed $j\geq 1$, suppose that $B(x,4r) \subset\subset\Omega_j$.  There exists a positive constant $C$, depending on $\Omega_j$, $B$, $B(x, 4r)$, $\Lambda$, $\lambda$, $m$, $M$, $p$ and $n$,  so that the following two estimates hold:
\begin{equation}\label{apriori1}\int_{B(x,r)} |\nabla u_k|^p +|u_k|^p dx \leq C \text{ for all }k>j,\end{equation}
and,
\begin{equation}\label{apriori2}\int_{B(x,r)} |\nabla (u_k^{p-1})|^p + |u_k|^{(p-1)p}dx \leq C \text{ for all }k>j.\end{equation}
\end{prop}

The key thing to note from Proposition \ref{locuniform} is this: For each fixed $j$, the estimates (\ref{apriori1}) and (\ref{apriori2}) are \textit{independent of $k$ for $k>j$}.

\subsection{Caccioppoli estimates on the approximating sequence.}  Let us fix $j\geq 1$ as in Proposition \ref{locuniform}.  In order to prove Proposition \ref{locuniform}, we work with three Caccioppoli  estimates.   In each estimate, we will make use of only one of the assumptions on $\sigma$, and so we make this explicit in the statement of the relevant lemma.  We will often suppress the dependence on $x$ in $\mathcal{A}_k$ and write $\mathcal{A}_k (\xi)$ instead of  $\mathcal{A}_k(\, \cdot ,\xi)$.
\begin{lem}\label{est1}  Suppose that (\ref{pformupper}) holds for $0<\lambda<1$.  There exists a constant $C=C(\lambda)>0$, such that for each $k>j$,
\begin{equation}\label{cacc1}
\int_{\Omega_j} |\nabla u_k |^p |h|^p dx \leq C\int_{\Omega_j} u_k^p |\nabla h|^p dx, \text{ for all }h\in C^{\infty}_0(\Omega_j).
\end{equation}
\end{lem}
\begin{proof}  Fix $j$ and $k>j$ as in the statement of the lemma, and let $v=u_k$.  With $h \in C^{\infty}_0(\Omega_j)$, $h\geq 0$, testing the weak formulation of (\ref{approxsol}) with $v h^p \in L^{1,p}_0(\Omega _j)$ yields
$$ \int_{\Omega_j}  \mathcal{A}_k(\nabla v) \cdot \nabla (v h^p)\, dx =  \int_{\Omega_j} h^p v^p d\sigma_k,$$
and hence,
\begin{equation}\label{cacc1step1}\begin{split}\nonumber
&\int_{\Omega_j}  \mathcal{A}_k(\nabla v)\cdot (\nabla v ) h^p \, dx = \int_{\Omega_j} h^p v^pd\sigma_k -\int_{\Omega_j} v    \mathcal{A}_k(\nabla v)\cdot \nabla (h^p) \, dx.
\end{split}\end{equation}
Applying Lemma \ref{mollem} and crudely employing (\ref{ellip}), we dominate the right hand side of this equality by
$$\lambda \int_{\Omega_j} \mathcal{A}_k(\nabla(hv))\!\cdot\!\nabla (h v)\, dx +pM\!\int_{\Omega_j} v \, |h|^{p-1}   |\nabla v|^{p-1} |\nabla h| dx.
$$
Recall Young's inequality with $\varepsilon$: for any $\varepsilon>0$, and $a,b\geq 0$,
\begin{equation}\label{young}ab \leq \varepsilon a^p + (p\varepsilon)^{-1/(p-1)}\frac{(p-1)}{p} b^{p'}.\end{equation}
It follows from (\ref{young}) that for any $\varepsilon>0$ there exists a constant $C(\varepsilon)$, depending on $\varepsilon$, $m$, $M$ and $p$, such that
$$pM\int_{\Omega_j} v  |h|^{p-1}  \! |\nabla v|^{p-1} \! |\nabla h| dx\leq \varepsilon m \int_{\Omega_j}|\nabla v|^p h^p \, dx + C(\varepsilon)\int_{\Omega_j} v^p |\nabla h|^p \, dx.
$$
Applying (\ref{ellip}) in the first term in the right hand side of this inequality, and bringing our estimates together, we obtain
\begin{equation}\label{afterellip}\begin{split}\int_{\Omega_j}  \mathcal{A}_k(\nabla v)&\cdot (\nabla v ) h^p \, dx \leq  \lambda\int_{\Omega_j} \mathcal{A}_k(\nabla(hv))\!\cdot\!\nabla (h v)\, dx \\
&+\varepsilon \int_{\Omega_j}[ \mathcal{A}_k(\nabla v)\cdot(\nabla v)] h^p \, dx + C(\varepsilon)\int_{\Omega_j} v^p |\nabla h|^p \, dx.
\end{split}\end{equation}
Now we raise both sides of (\ref{afterellip}) to the power $1/p$ and appeal to the elementary inequality
\begin{equation}\label{elemabineq}(a+b)^{1/p}\leq a^{1/p}+ b^{1/p} \text{ for } a,b>0.
\end{equation}After applying the Minkowski inequality (\ref{mink}), we arrive at
\begin{equation}\begin{split}\nonumber\Bigl(\int_{\Omega_j}& [\mathcal{A}_k(\nabla v)\cdot \nabla v] h^p \, dx\Bigl)^{1/p} \leq  (\lambda^{1/p}\! +\! \varepsilon^{1/p})\Bigl(\int_{\Omega_j}[ \mathcal{A}_k(\nabla v)\cdot\nabla v] h^p \, dx\Bigl)^{1/p}\\
&+ \lambda^{1/p}\Bigl(\int_{\Omega_j}[ \mathcal{A}_k(\nabla h)\cdot\nabla h] v^p \, dx\Bigl)^{1/p}+ \Bigl(C(\varepsilon)\int_{\Omega_j} v^p |\nabla h|^p \, dx\Bigl)^{1/p}.
\end{split}\end{equation}
 Choosing $\varepsilon <(1-\lambda^{1/p})^p$ and rearranging, we conclude that
$$\int_{\Omega_j} [\mathcal{A}_k(\nabla u_k)\cdot \nabla u_k] h^p dx \leq C(\lambda)\int_{\Omega_j} u_k^p |\nabla h|^p dx.
$$
Appealing to (\ref{ellip}), we obtain (\ref{cacc1}).
\end{proof}

\begin{lem}\label{est3}  Suppose that (\ref{pformlower}) holds for some $\Lambda>0$.  There exists a constant $C=C(\Lambda)>0$, such that for all $k>j$, one has
\begin{equation}\label{cacc3}
\int_{\Omega_j} \frac{|\nabla u_k |^p}{u_k^p} |h|^p dx \leq C\int_{\Omega_j} |\nabla h|^p dx, \text{ for all }h\in C^{\infty}_0(\Omega_j).
\end{equation}
\end{lem}
\begin{proof}
Let $h\in C^{\infty}_0(\Omega_j)$, with $h\geq 0$.  Since $u_j$ satisfies the Harnack inequality in $\Omega_j$, there exists a constant $c>0$ so that $u_k >c$ on the support of $h$.  Thus $h^p u_k^{1-p}\in \text{L}^{1,p}_0(\Omega_j)$ is a valid test function, and hence
\begin{equation}\label{gradlogest2}-\int_{\Omega_j}  \mathcal{A}_k(\nabla u_k) \cdot \nabla \Bigl(\frac{h^p}{u_k^{p-1}}\Bigl) dx = -\int_{\Omega_j} h^pd\sigma_k.
\end{equation}
On the other hand, by differentiating and applying (\ref{ellip}) we see that
\begin{equation}\begin{split}\label{gradlogest6}(p-1)\!\int_{\Omega_j} \frac{\mathcal{A}_k(\nabla u_k)\cdot\nabla u_k}{u_k^p} &h^p dx \leq  -\int_{\Omega_j}  \mathcal{A}_k(\nabla u_k)\!\cdot \!\nabla \Bigl(\frac{h^p}{u_k^{p-1}}\Bigl) dx\\
& + Mp\int_{\Omega_j} \frac{|\nabla u_k|^{p-1}}{u_k^{p-1}} |\nabla h| h^{p-1} dx.
\end{split}\end{equation}
The second term on the right may be estimated using Young's inequality and (\ref{ellip}): for each $\varepsilon>0$, there exists $C(\varepsilon)>0$ such that
\begin{equation}\begin{split}\label{gradlogest7}p\!\int_{\Omega_j} \frac{|\nabla u_k|^{p-1}}{u_k^{p-1}} |\nabla h| h^{p-1} dx \leq &\varepsilon \int_{\Omega_j} \frac{\mathcal{A}_k(\nabla u_k)\cdot\nabla u_k}{u_k^p}h^p dx \\
&+ C(\varepsilon)\!\int_{\Omega_j} |\nabla h|^p dx.
\end{split}\end{equation}
Applying (\ref{gradlogest2}) and (\ref{gradlogest7}) into (\ref{gradlogest6}), we estimate
\begin{equation}\label{gradlogest3}\nonumber(p-1-\varepsilon)\int_{\Omega_j} \frac{\mathcal{A}_k(\nabla u_k)\cdot\nabla u_k}{u_k^p} h^p dx \leq  -\int_{\Omega_j} h^pd\sigma_k + C(\varepsilon)\int_{\Omega_j} |\nabla h|^p dx.
\end{equation}
To bound the first term on the right hand side of this inequality, note that combining Lemma \ref{mollem} with the lower form bound (\ref{pformlower}) yields
\begin{equation}\label{molldiscacc3}
-\int_{\Omega_j}h^pd\sigma_k  \leq \Lambda\int_{\Omega_j} \mathcal{A}_{k}(\nabla h)\cdot\nabla h dx \leq M\Lambda \int_{\Omega_j} |\nabla h|^p dx.
\end{equation}
Substituting (\ref{molldiscacc3}) into the penultimate inequality, we deduce (\ref{cacc3}) from (\ref{ellip}).
\end{proof}

The third lemma will only be used in the case $p\geq2$, but is valid for all $1<p<\infty$.

\begin{lem}\label{est2}  Suppose that (\ref{pformupper}) holds with $0<\lambda<(p-1)^{2-p}$.  There exists a constant $C=C(\lambda )>0$, such that for all $k>j$,
\begin{equation}\label{cacc2}
\int_{\Omega_j} |\nabla (u_k)^{p-1} |^p |h|^p dx \leq C\int_{\Omega_j} |(u_k)^{p-1}|^p |\nabla h|^p dx, \text{ for all } h\in C^{\infty}_0(\Omega_j).
\end{equation}
\end{lem}
\begin{proof}  Fix $k\geq j$ and $h\in C^{\infty}_0(\Omega_j)$.   Let $v = u_k$, and note that
\begin{equation}\begin{split}\label{3rdcacctest}\int_{\Omega}  [\mathcal{A}_k(\nabla v)&\cdot\nabla v] v^{(p-2)p}h^p dx  = \frac{1}{(p-1)^2}\int_{\Omega}\mathcal{A}_k(\nabla v)\cdot \nabla \bigl(v^{(p-1)^2}h^p\bigl)dx\\
&\qquad -\frac{p}{(p-1)^2}\int_{\Omega} [\mathcal{A}_k(\nabla v)\cdot \nabla h ]v^{(p-1)^2}h^{p-1} dx.
\end{split}\end{equation}
Using the properties of $v$ (see Lemma \ref{existencelem}), $v^{(p-1)^2}h^p$ is a valid test function for all $p>1$, and hence
\begin{equation}\label{vp1testest}\begin{split} \int_{\Omega}\mathcal{A}_k(\nabla v)&\cdot \nabla \Bigl(v^{(p-1)^2}h^p\Bigl)dx  = \int_{\Omega} v^{p(p-1)}h^p d\sigma_k \\
&\leq \lambda\int_{\Omega} \mathcal{A}_k(\nabla (v^{p-1}h))\cdot\nabla(v^{p-1}h) dx,
\end{split}\end{equation}
where Lemma \ref{mollem} has been applied in the second inequality.  The Minkowski inequality implies that
\begin{equation}\begin{split}\nonumber\Bigl(\int_{\Omega} \mathcal{A}_k(\nabla (v^{p-1}h))\cdot\nabla(v^{p-1}h) dx\Bigl)^{1/p}&\leq  \Bigl(\int_{\Omega} [\mathcal{A}_k(\nabla  v^{p-1} )\cdot\nabla v^{p-1} ] h^pdx\Bigl)^{1/p}\\
&+ \Bigl(\int_{\Omega} [\mathcal{A}_k(\nabla h)\cdot\nabla h ] v^{p(p-1)}dx\Bigl)^{1/p}.
\end{split}\end{equation}
Note that $[\mathcal{A}_k(\nabla v^{p-1})\cdot\nabla v^{p-1} ]h^p = (p-1)^p [\mathcal{A}_k(\nabla v)\cdot\nabla v ]v^{p(p-2)}h^p$.  Bringing our estimates together, making use of the boundedness of $\mathcal{A}$ from (\ref{ellip}),  we see that \begin{equation}\begin{split}\label{vp1testest2} \Bigl(\int_{\Omega} &(\mathcal{A}_k(\nabla v)\cdot\nabla v )v^{(p-2)p}h^p dx\Bigl)^{1/p} \\
&\leq \lambda^{(1/p)}(p-1)^{\tfrac{p-2}{p}}\Bigl(\int_{\Omega}(\mathcal{A}_k(\nabla v)\cdot\nabla v )v^{(p-2)p} h^p dx\Bigl)^{1/p} \\
&+ \lambda^{1/p}\Bigl(CM\int_{\Omega} v^{p(p-1)}|\nabla h|^p dx\Bigl)^{1/p}\\
& + \Bigl(CM \int_{\Omega} |\nabla v|^{p-1}v^{(p-2)p+1}|\nabla h| h^{p-1} dx\Bigl)^{1/p}.
\end{split}\end{equation}
The third term in the right hand side of (\ref{vp1testest2}) is handled with Young's inequality: for any $\varepsilon>0$, there exists $C(\varepsilon)$ such that
\begin{equation}\begin{split}\nonumber\label{vp1testest3} \Bigl( &CM \int_{\Omega} |\nabla v|^{p-1} v^{(p-2)p+1}|\nabla h| h^{p-1} dx\Bigl)^{1/p}\\
&\leq \varepsilon \Bigl(\int_{\Omega_j}(\mathcal{A}_k(\nabla v)\cdot \nabla v)v^{(p-2)p}h^p dx\Bigl)^{1/p} +C(\varepsilon)\Bigl(\int_{\Omega_j}v^{p(p-1)}|\nabla h|^p dx\Bigl)^{1/p}.
\end{split}\end{equation}
Here (\ref{ellip}) has also been used (as in (\ref{afterellip})).  By assumption on $\lambda$, we have $\lambda^{(1/p)}(p-1)^{1-2/p}<1$. Choose $\varepsilon>0$ so that $\varepsilon<1- \lambda^{(1/p)}(p-1)^{1-2/p}$.  Applying the previous estimate into (\ref{vp1testest2}) and rearranging, we conclude that
$$\Bigl(\int_{\Omega} [\mathcal{A}_k(\nabla v)\cdot\nabla v ]v^{(p-2)p}h^p dx\Bigl)^{1/p}\leq C(\lambda) \Bigl(\int_{\Omega} v^{p(p-1)}|\nabla h|^p dx\Bigl)^{1/p}.
$$
Appealing to (\ref{ellip}) once again, the lemma is proved.
\end{proof}

\subsection{A uniform gradient estimate: the proof of Proposition \ref{locuniform}}  Having established the required Caccioppoli inequalities, we move onto proving Proposition \ref{locuniform}.
\begin{proof}[The proof of Proposition \ref{locuniform}]  Fix $k>j$, and let $v = u_k^q$ with $q = \max (p-1,1)$.  To prove (\ref{apriori1}) and (\ref{apriori2}), we will employ Proposition \ref{locdoub} in $U = \Omega_j$ to show that $v^p$ is doubling in $\Omega_j$, with constants independent of $k$.  To verify the hypothesis of Proposition \ref{locdoub}, we first show that $v^p$ satisfies a weak reverse H\"{o}lder inequality, i.e. that (\ref{wrh}) holds in $\Omega_j$.  To this end, let us fix $B(z,2s)\subset\subset \Omega_j$.

First suppose $1<p<n$.  For any $\psi \in C^{\infty}_{0}(\Omega_j)$, an application of Sobolev's inequality yields
\begin{equation}\label{sobineq}\begin{split}
\Bigl(\int_{\Omega_j}v^{\frac{pn}{n-p}} |\psi|^{\frac{pn}{n-p}}dx\Bigl)^{\frac{n-p}{pn}}\leq C& \Bigl(\int_{\Omega_j} |\nabla v|^p|\psi|^p \, dx\Bigl)^{1/p}\\&+ C\Big(\int_{\Omega_j} v^p \, |\nabla\psi|^p \, dx\Bigl)^{1/p}.
\end{split}\end{equation}
Applying Lemma \ref{est1} (if $p\leq2$) or Lemma \ref{est2} (if $p\geq 2$) in the first term on the right hand side of (\ref{sobineq}), we deduce that
\begin{equation}\label{caccsob}\Bigl(\int_{\Omega_j}v^{\frac{pn}{n-p}} |\psi|^{\frac{pn}{n-p}}dx\Bigl)^{\frac{n-p}{n}} \leq C(\lambda) \int_{\Omega_j}v^{p}|\nabla \psi|^p dx.
\end{equation}
Specialising (\ref{caccsob}) to the case when $\psi\in C^{\infty}_0(B(z, 2s))$, with $\psi\equiv1 $ in $B(z,s)$, and $|\nabla\psi|\leq C/s$, we have
\begin{equation}\label{lpest}\Bigl(\dashint_{B(z,s)}(v^p)^{\frac{n}{n-p}}dx\Bigl)^{\frac{n-p}{n}} \leq  C(\lambda)\dashint_{B(z,2s)}v^{p} \, dx.
\end{equation}
Hence the weak reverse H\"{o}lder inequality (\ref{wrh}) holds in $U = \Omega_j$, with $w=v^p$ and $q=n/(n-p)$.

In the case when $p= n$, we appeal to the following Sobolev inequality: for each $q<\infty$, and for all $f\in C^{\infty}_0(B(z,2s))$,
\begin{equation}\label{sobpeq2}
\Bigl(\dashint_{B(z,2s)} |f(y)|^{q} \, dy \Bigl)^{1/q} \leq C(q)\Bigl(\int_{B(z,2s)}|\nabla f(y)|^p \, dy\Bigl)^{1/p}.
\end{equation}
see for example \cite{MZ97}, Corollary 1.57.
Using (\ref{sobpeq2}) in (\ref{sobineq}), and following the above argument to display (\ref{lpest}), we see that for each $q<\infty$,  (\ref{wrh}) holds in $U = \Omega_j$, with $w=v^p$.  When $p>n$, standard Sobolev inequalities show that (\ref{wrh}) continues to hold in $U=\Omega_j$, with $w=v^p$ and any $q\leq \infty$.

To apply Proposition \ref{locdoub}, it remains to  show that $\log(v)\in BMO(\Omega_j)$. For this, fix a ball $B(z, 2s)\subset \Omega_j$, and note that Lemma \ref{est3} implies
\begin{equation}\label{locgradmorest}\int_{B(z,s)} \frac{|\nabla u_k|^p}{u_k^p} dx \leq C(\Lambda) s^{n-p}.
\end{equation}
Indeed, to prove display (\ref{locgradmorest}) one simply picks $h\in C^{\infty}_0(B(z,2s))$ so that $h\equiv 1$ on $B(z,s)$ and $|\nabla h|\leq C/s$ in display (\ref{cacc3}).  On the other hand,
using the Poincar\'{e} inequality yields
\begin{equation}\begin{split}\label{poincare1}
\dashint_{B(z,s)} \!|\log v & -\dashint_{B(z,s)}\! \log v|^p dx\leq Cs^{p-n}\!\int_{B(z,s)} \!\frac{|\nabla u_k|^p}{u_k^p}  dx\leq \!C(\Lambda),
\end{split}\end{equation}
and hence $\log v \in BMO(\Omega_j)$, with $BMO$-norm depending only on $p, \Lambda, m$ and $M$ (see (\ref{BMOdef})).  In particular, $v^p$ satisfies both (\ref{wrh}) and (\ref{bmo1}) in $\Omega_j$.  Proposition \ref{locdoub} can now be applied to conclude that $v^p$ is doubling in $\Omega_j$, with doubling constant depending on $m$, $M$, $n$, $p$, $\lambda$ and $\Lambda$, see (\ref{doub}).  In other words, there exists a constant $C=C(\lambda, \Lambda)$, such that for each ball $B(z, 4s)\subset \Omega_j$ one has
\begin{equation}\label{locdoubref}
\dashint_{B(z,2s)} v^p dx \leq C \dashint_{B(z,s)} v^p dx.
\end{equation}

Since $\Omega_j$ is a connected set with smooth boundary, one can find a Harnack chain from $B(x,2r)$ to the fixed ball $B\subset\subset \Omega_1$.  In other words, there are three positive constants $c_0,\,c_1$ and $N>0$, depending on the smooth parameterization of $\Omega_j$, along with points $x_0, \dots x_N$ and balls $B(x_i, 4r_i)\subset \Omega_j$ satisfying
\begin{enumerate}
\item $B(x_0, r_0) = B(x,2r)$, and $B(x_N, r_N) = B$;
\item $r_i\geq c_0 \min(r_0, r_N)$, and $|B(x_i, r_i)\cap B(x_{i+1}, r_{i+1})| \geq c_1 \min(r_0, r_N)^n$ for all $i = 0\dots N-1$.
\end{enumerate}
Combining the Harnack chain with the property that $v^p$ is doubling in $\Omega_j$, a Harnack chain argument yields
\begin{equation}\nonumber\dashint_{B(x,2r)} v^p dx \leq C(B(x, r), \Omega_j, B, \lambda, \Lambda) \dashint_B v^p dx.
\end{equation}
By the normalization on $v^p$ (recall (\ref{approxsol})), we get
\begin{equation}\label{vlpest}\dashint_{B(x,2r)} v^p dx \leq C(B(x, r), \Omega_j, B, \lambda, \Lambda).
\end{equation}
To complete the proof, it remains to deduce the required bounds for the gradient in (\ref{apriori1}) and (\ref{apriori2}).  First suppose that $p\geq2$.  In this case, we combine Lemmas \ref{est1} and \ref{est3} with (\ref{vlpest}) to conclude that the following two estimates hold:
$$\int_{B(x, r)}|\nabla u_k|^p dx \leq \frac{C}{r^{p-n}}\Bigl(\dashint_{B(x,2r)} v^p dx\Bigl)^{1/q}\leq C,$$
and
$$\int_{B(x, r)}|\nabla u_k^{p-1}|^p dx \leq \frac{C}{r^{p-n}}\Bigl(\dashint_{B(x,2r)} v^p dx\Bigl)^{1/q}\leq C,$$
for a constant $C>0$, depending on $n$, $p$, $m$, $M$, $B$, $\Lambda$, $\lambda$, $\Omega_j$ and $B(x, r)$.  Here we have used H\"{o}lder's inequality in the first of the two displays above.

In the case $1<p<2$, note that combining Lemma \ref{est1} with (\ref{vlpest}), we have
$$\int_{B(x, r)}|\nabla u_k|^p dx \leq \frac{C}{r^{p-n}}\dashint_{B(x,2r)} v^p dx \leq C,$$
 for a positive constant $C>0$ depending on $n$, $p$, $m$, $M$, $B$, $\Lambda$, $\lambda$, $\Omega_j$ and $B(x, r)$.  On the other hand, a simple consequence of Lemma \ref{est3} is the inequality
$$\int_{B(x,r)}\frac{|\nabla u_k|^p}{u_k^p} dx \leq C,
$$
(cf. display (\ref{locgradmorest})).  One can readily interpolate between these two estimates to yield (\ref{apriori2}), indeed
\begin{equation}\begin{split}\int_{B(x,r)} |\nabla u_k|^p& u_k^{p(p-2)} dx  \leq \int_{B(x,r)\cap \{u_k\geq 1\}} |\nabla u_k|^p u_k^{p(p-2)} dx \\
& + \int_{B(x,r)\cap \{u_k\leq 1\}} |\nabla u_k|^p u_k^{p(p-2)} dx\\
& \leq \int_{B(x,r)} |\nabla u_k|^p dx + \int_{B(x,r)} \frac{|\nabla u_k|^p}{u_k^p}dx \leq C,
\end{split}\end{equation}
with $C$ depending on $n$, $p$, $m$, $M$, $B$, $\Lambda$, $\lambda$, $\Omega_j$ and $B(x, r)$ (but independent of $k$).
\end{proof}

\subsection{Convergence to a solution}\label{passtolimit}  Our first task is to deduce the existence of a solution $u^{(j)}$ of (\ref{schro}) in each $\Omega_j$.  We will concentrate on the argument in $\Omega_1$ for ease of notation.

From (\ref{apriori1}) and (\ref{apriori2}), it follows by choosing a suitable covering of $\Omega_1$ that there is a constant $K=K(\lambda, \Lambda, \Omega_1, B)$ so that for each $k\geq 2$, we have
\begin{equation}\label{omega1uniformbd}\begin{split}&\int_{\Omega_1} \! \left ( |\nabla u_k|^p \!+\!|u_k|^p \right ) dx\!\leq\! K,\! \text{ and } \!\int_{\Omega_1}\! \left ( |\nabla (u_k)^{p-1}|^p \!+\! |u_k|^{p(p-1)} \right) dx\!\leq\! K.
\end{split}\end{equation}
Using weak compactness of $W^{1,p}(\Omega_1)$, we claim that there is a subsequence $u_{j,1}$ of $u_j$, and a limit function $u^{(1)}\in W^{1,p}(\Omega_1)$ satisfying the following properties:
\begin{enumerate}
\item $u_{j,1} \rightarrow u^{(1)}$ weakly in $W^{1,p}(\Omega_1)$,
\item $u_{j,1}^{p-1}\rightarrow (u^{(1)})^{p-1}$ weakly in $W^{1,p}(\Omega_1)$,
\item $u_{j,1} \rightarrow u^{(1)}$ a.e. in $\Omega_1$.
\item $u_{j,1} \rightarrow u^{(1)}$ in $L^{pq}(\Omega_1)$, where $q=\max(p-1,1)$
\end{enumerate}
Indeed, from (\ref{omega1uniformbd}) and weak compactness, we first pass to a subsequence satisfying (1).  Appealing to Rellich's theorem, we obtain a further subsequence $u_{j,1}$ satisfying $u_{j,1}\rightarrow u^{(1)}$ in $L^p(\Omega)$, and also property (3).  But then $u_{j,1}^{p-1}$ converges almost everywhere to $(u^{(1)})^{p-1}$ in $\Omega_1$.  Since $u_{j,1}^{p-1}$ is uniformly bounded in $W^{1,p}(\Omega_1)$, it follows from standard Sobolev space theory (see Theorem 1.32 of \cite{HKM}) that we may pass to a further subsequence so that (2) holds.  If $1<p\leq2$ then the property (4) has already been demonstrated.  If $p> 2$, then a final application of Rellich's theorem to the sequence $u_{j,1}^{p-1}$ yields the required $L^{p(p-1)}(\Omega_1)$ convergence for a subsequence.

Let $h\in C^{\infty}_0(\Omega_1)$, and let $U\subset\subset\Omega_1$ be an open set containing $\text{supp}(h)$.  Recall that $\sigma \in L^{-1,p'}(U)$, from which it follows that
\begin{equation}\label{sigmaconv}
\langle \sigma_{j,1}, u_{j,1}^{p-1}h\rangle \rightarrow \langle \sigma, (u^{(1)})^{p-1}h\rangle, \text{ as }j\rightarrow \infty.
\end{equation}
Indeed, by the triangle inequality we write
\begin{equation}\begin{split}\nonumber|\langle \sigma_{j,1}, u_{j,1}^{p-1}h\rangle - \langle \sigma, (u^{(1)})^{p-1}h\rangle| \leq |\langle & \sigma, (u_{j,1}^{p-1}-(u^{(1)})^{p-1})h\rangle| \\
&+ |\langle (\sigma_{j,1} - \sigma), u_{j,1}^{p-1} h\rangle|.
\end{split}\end{equation}
The first term on the right hand side converges to zero on account of the weak convergence property (2).  For the second term, we estimate
$$|\langle (\sigma_{j,1} - \sigma), u_{j,1}^{p-1} h\rangle|\leq ||\nabla (u_{j,1}^{p-1}h)||_{L^p(U)}||\sigma_{j,1}-\sigma||_{L^{-1,p'}(U)}.
$$
The right hand side here convergences to zero due to standard properties of the mollification, since the first term is bounded due to (\ref{omega1uniformbd}).  This establishes (\ref{sigmaconv}).

We next claim that there is another subsequence of $u_{j,1}$ (again denoted by $u_{j,1}$) such that
\begin{equation}\label{gradconv}
\mathcal{A}_{j,1}(\cdot, \nabla  u_{j,1})\rightarrow \mathcal{A}(\cdot,\nabla u^{(1)})\text{ in }(L^1_{\text{loc}}(\Omega_1))^n.
\end{equation}
The proof of (\ref{gradconv}) will be quite involved.  For this reason we postpone the proof to Section \ref{convmeassec} and complete the rest of the argument.

From (\ref{sigmaconv}) and (\ref{gradconv}), it follows that
\begin{equation}\label{limit1}
-\text{div}(\mathcal{A}(\nabla u^{(1)})) = \sigma (u^{(1)})^{p-1} \text{ in }\mathcal{D}'(\Omega_1).
\end{equation}
Here the dominated convergence theorem has been used on the left hand side, in conjunction with (\ref{gradconv}).   By the the normalization of the sequence $(u_j)_j$ in (\ref{approxsol}) and property (4), we see that $\int_B (u^{(1)})^{qp}dx = 1, \text{ with }q = \max(p-1,1).$

The argument is now repeated in each $\Omega_k$ for $k>1$.  Each time we choose a subsequence $(u_{j,k})_j$ of the sequence $(u_{j,k-1})_j$ converging to $u^{(k-1)}$ in $\Omega_{k-1}$.  In this manner we arrive at functions $u^{(k)}$ satisfying
\begin{equation}\label{limitk}
-\text{div}(\mathcal{A}(\nabla u^{(k)})) = \sigma (u^{(k)})^{p-1} \text{ in }\mathcal{D}'(\Omega_k),
\end{equation}
and
\begin{equation}\label{uk1ball}\int_B (u^{(k)})^{qp}dx = 1, \text{ with }q = \max(p-1,1).
\end{equation}
Note that $u^{(k)} = u^{(k-1)}$ in $\Omega_{k-1}$ (equality here holding in the sense of $W^{1,p}(\Omega_{k-1})$ functions), which holds since the sequence $u_{j,k}$ converges weakly to both $u^{(k-1)}$ and $u^{(k)}$ in $W^{1,p}(\Omega_{k-1})$.  Hence if we define $u$ by  $u = u^{(k)}$ in $\Omega_k$, then $u$ is well defined and $-\text{div}(\mathcal{A}(\nabla u)) = \sigma u^{p-1}\text{ in }\Omega.$
From (\ref{uk1ball}) it follows that $u$ is not the zero function.

Recall that for each $k>j$, the approximate solution $u_{k}^{qp}$ is doubling in $\Omega_j$ with doubling constants independent of $k$ (see (\ref{locdoubref})).   Passing to the limit (using property (4)) it follows that $u^{qp}$ is locally doubling in $\Omega$.  In particular $u>0$ almost everywhere in $\Omega$, and hence $\log(u)$ is well defined almost everywhere.

 We shall now show that (\ref{mult1}) holds.  Fix $k\geq 1$.  Then, for each $j>k$, $\log(u_{j,k})\rightarrow \log (u)$ a.e. in $\Omega_k$.  Combining Lemma \ref{est3} with Theorem 1.32 of \cite{HKM}, we pass to a subsequence of $u_{j,k}$ whose logarithm converges weakly in $W^{1,p}(\Omega_k)$ to $\log (u)$.  From the lower-weak semicontinuity of $L^p(\Omega_k)$, it now follows that
\begin{equation}\label{multinomegak}\int_{\Omega} \frac{|\nabla u|^p}{u^p} |h|^p dx \leq C(\Lambda)\int_{\Omega}|\nabla h|^p, \text{ for all }h\in C^{\infty}_0(\Omega_k).
\end{equation}
Since there is no dependence on $k$ in constant appearing in (\ref{multinomegak}), we let $k\rightarrow \infty$ to deduce (\ref{mult1}).

Save for the estimate (\ref{gradconv}) (which will be proved in Section \ref{convmeassec}), to finish the proof of part (i) of Theorem \ref{mainthm} it remains to show that $v = \log(u)$ is a solution of (\ref{genric}) satisfying (\ref{mult2}).  This is the content of the following lemma:

\begin{lem}\label{logsublin}   Let $\Omega$ be an open set, and suppose that $\sigma\in L^{-1,p'}_{\rm{loc}}(\Omega)$.   If there exists a positive solution $u$ of (\ref{genschro}) satisfying (\ref{mult1}), then $v = \log(u)\in L^{1,p}_{\rm{loc}}(\Omega)$ is a solution of (\ref{genric}) and (\ref{mult2}) holds.
\end{lem}
\begin{proof}
Let $\varepsilon>0$.  Then for $h\in C^{\infty}_0(\Omega)$, test the weak formulation of (\ref{schro}) with $\psi = h (u+\varepsilon)^{1-p}\in L^{1,p}_c(\Omega)$.  This yields
\begin{equation}\label{Schro2Ric}\int_{\Omega} \frac{\mathcal{A}(\,\cdot, \nabla u)}{(u+\varepsilon)^{p-1}} \cdot \nabla h \,dx =(p-1)\int_{\Omega} \frac{\mathcal{A}(\,\cdot, \nabla u)\cdot\nabla u}{(u+\varepsilon)^p} h dx + \langle \sigma \frac{u^{p-1}}{(u+\varepsilon)^{p-1}}, h\rangle.
\end{equation}
Letting $\varepsilon\rightarrow 0$, it follows from the condition (\ref{mult1}), and the dominated convergence theorem, that we have
$$\int_{\Omega} \frac{\mathcal{A}(\cdot,\nabla u)}{(u+\varepsilon)^{p-1}} \cdot \nabla h \,dx \rightarrow \int_{\Omega} \frac{\mathcal{A}(\,\cdot, \nabla u)}{u^{p-1}} \cdot \nabla h \,dx, \text{ and}$$
$$\int_{\Omega} \frac{\mathcal{A}(\,\cdot, \nabla u)\cdot\nabla u}{(u+\varepsilon)^p} h dx\rightarrow \int_{\Omega} \frac{\mathcal{A}(\,\cdot, \nabla u)\cdot\nabla u}{u^p} h dx.
$$
To handle the last term in (\ref{Schro2Ric}),  note that
$$\nabla \Bigl(\frac{u}{u+\varepsilon}\Bigl)^{p-1} = (p-1)\frac{\nabla u}{u}\Bigl(\frac{ u^{p-1}}{(u+\varepsilon)^{p-1}}\Bigl)\frac{\varepsilon}{u+\varepsilon}. 
$$
Hence $|\nabla \bigl(\tfrac{u}{u+\varepsilon}\bigl)^{p-1}|^p\leq (p-1)^p\bigl|\frac{\nabla u}{u}\bigl|^p$, and the right hand side here is in $L^1_{\text{loc}}(\Omega)$ on account of (\ref{mult1}).  Since $|\nabla \bigl(\tfrac{u}{u+\varepsilon}\bigl)^{p-1}|\rightarrow 0$ whenever $u>0$, the dominated convergence theorem yields $|\nabla \bigl(\frac{u}{u+\varepsilon}\bigl)^{p-1}|\rightarrow 0$ in $L^p_{\text{loc}}(\Omega)$ as $\varepsilon \rightarrow 0$.

On the other hand, it is clear that  $\frac{u}{u+\varepsilon} \rightarrow 1 \text{ in }L^p_{\text{loc}}(\Omega) \text{ as } \varepsilon \rightarrow 0,$ and therefore
$\tfrac{u}{u+\varepsilon} \rightarrow 1 \text{ in }W^{1,p}_{\text{loc}}(\Omega) \text{ as } \varepsilon \rightarrow 0.$
Since $\sigma \in L^{-1,p'}(\text{supp}(h))$, we conclude that
$$\langle \sigma \Bigl(\frac{u}{u+\varepsilon}\Bigl)^{p-1}, h\rangle \rightarrow\langle \sigma, h\rangle, \text{ as } \varepsilon \rightarrow 0.
$$
It follows that $v = \log(u)$ is a solution of (\ref{genric}).  The estimate (\ref{mult2}) is immediate from (\ref{mult1}).\end{proof}

\subsection{Convergence in measure}\label{convmeassec}
To complete the proof of part (i) of Theorem \ref{mainthm}, it remains to prove (\ref{gradconv}) for a subsequence of $(u_{j,1})_j$. Following a well known reduction, see for example Theorem 6.1 of \cite{BBGVP}, it suffices to assert a convergence in measure result.  First note that
$$\int_{\Omega_1} |\mathcal{A}_{\varepsilon_j}(\nabla u_{j,1}) - \mathcal{A}(\nabla u_{j,1})|dx \leq \omega(\varepsilon_j)\int_{\Omega_1} |\nabla u_{j,1}|^{p-1}dx\rightarrow 0, \text{ as }j\rightarrow \infty,
$$
where in the last line we are using (\ref{apriori1}) and (\ref{cont}).  As a result, in order to assert (\ref{gradconv}) it suffices to prove (for a suitable subsequence of $(u_{j,1})_j$) that
$$\mathcal{A}(\cdot, \nabla u_{j,1}) \rightarrow \mathcal{A}(\cdot, \nabla u^{(1)}) \text{ in }L^1_{\text{loc}}(\Omega_1).
$$
From the Vitali convergence theorem and the gradient estimate (\ref{apriori1}), this local $L^1$ convergence will follow once we assert that $\mathcal{A}(\cdot,\nabla u_{j,1})$ converges locally in measure to $\mathcal{A}(\cdot, \nabla u^{(1)})$ in $\Omega_1$.  Due to the continuity of the operator $\mathcal{A}$, this in turn is a consequence of the following lemma:
\begin{lem}\label{convmeas} Suppose $B_{2r} = B(x,2r)\subset \Omega_1$.  Then for every $\delta>0$, we have
$$|\{x\in B_r\, :\, |\nabla u_{j,1} - \nabla u_{k,1}|> \delta \}| \rightarrow 0 \text{ as }j,k\rightarrow \infty,
$$
\end{lem}

Note that this reduction is still valid without the continuity assumption on $\mathcal{A}$.  In this case one instead appeals to Nemitskii's theorem, as in \cite{BBGVP}, p. 259.

\begin{proof}  Let $\delta>0$.  To simplify notation put $u_{j,1} = \tilde{u}_j$, and $u^{(1)} = \tilde{u}$. We introduce parameters $A$ and $\mu$ satisfying $A>1$ and $0<\mu<A/2$, and write
$$|\{x\in B_r\, :\, |\nabla \tilde{u}_{j} - \nabla \tilde{u}_{k}|>\delta\}|\leq I  +II + III + IV,
$$
where
$$ I = |\{x\in B_r\,:\, |\nabla \tilde{u}_j| >A\}| + |\{x\in B_r\, : \, |\nabla \tilde{u}_k|>A\}|,
$$
$$II  =  |\{x\in B_r\,:\, \tilde{u}_j >A\}| + |\{x\in B_r\, : \, \tilde{u}_k>A\}|,
$$
$$III = |\{x\in B_r \, :\, |\tilde{u}_j -\tilde{u}_k| >\mu\}|,$$
and  $IV = |E|$, with $E$ defined by
\begin{equation}\begin{split}E = \{x\in B_r \, :\, &|\nabla \tilde{u}_j - \nabla \tilde{u}_k|>\delta, \, |\tilde{u}_j -\tilde{u}_k| \leq\mu; \,  |\nabla \tilde{u}_j| \leq A;\\
&\, |\nabla \tilde{u}_k|\leq A;\, \tilde{u}_j <A,\, \tilde{u}_k<A\}.
\end{split}\end{equation}
It is the estimate for $IV$ which will require a careful analysis.  We claim that there exists a constant $C(A, \delta)>0$, depending on $A$, $\delta$, $B(x,r)$, $\Omega_1$, the constant $K$ from (\ref{omega1uniformbd}), as well as $M$, $m$, $n$ and $p$, such that
\begin{equation}\label{convmeasgoal}IV\leq C(A,\delta)\cdot\bigl[ \mu^{\min(1, p-1)} + o(1)\bigl] \text{ as }j,k\rightarrow \infty.
\end{equation}
(we write $C(A,\delta)$ to emphasize the dependence on $A$ and $\delta$).

To show that this estimate will prove the lemma, let $\varepsilon>0$. First pick $A>1$ such that
$I + II \leq \varepsilon/4.$
Such a choice is possible by the uniform integrability estimate (\ref{omega1uniformbd}) and Chebyshev's inequality.

Next (with $A>1$ fixed), let us pick $\mu\in (0, A/2)$ and $N_1 \in \mathbf{N}$ so that if $j,k>N_1$ then $IV \leq \varepsilon/4.$  Here we have used the claimed estimate (\ref{convmeasgoal}).

With $\mu>0$ now fixed, the almost everywhere convergence of $\tilde{u}_j$ to $\tilde{u}$ yields $N\in \mathbf{N}$ with $N\geq N_1$ such that $III\leq \varepsilon /2$ for  every $j , k>N$.

We conclude that $|\{x\in B_r\, :\, |\nabla \tilde{u}_{j} - \nabla \tilde{u}_{k}|>\delta\}|\leq\varepsilon$ for $j,k>N$, as required.

It remains to prove (\ref{convmeasgoal}).  To this end, let $k,j>1$, and split $E$ into the two sets $E_1 = E\cap \{\tilde{u}_j \geq \tilde{u}_k\}$, and $E_2 = E\backslash E_1$.  We will shall prove (\ref{convmeasgoal}) with $E$ replaced by $E_1$.  The estimate for $E_2$ will follow analogously.  First note that from the properties of $E$, along with monotonicity assumption (\ref{monotone}), it follows that there is a positive constant $c(A, \delta)$ such that
$$c(\delta, A) \leq \bigl[\mathcal{A}(\nabla \tilde{u}_j) - \mathcal{A}(\nabla \tilde{u}_k)\bigl]\cdot \nabla (\tilde{u}_j - \tilde{u}_k)(x), \text{ for each }x\in E_1.$$

Let  $h\in C^{\infty}_0(B_{2r})$ be a nonnegative bump function satisfying $h\equiv 1$ on $B_r$, and $|\nabla h|\leq C$ (the constant here depends on $r$, but this is suppressed as the constant in (\ref{convmeasgoal}) may depend on $r$).  Since both $u_j\leq A$ and $u_k\leq A$ in $E$, the previous inequality yields
$$IV\! \leq\! \frac{c(\delta, A)}{A}\!\int_{E_1} \bigl[(\mathcal{A}(\nabla \tilde{u}_j) - \mathcal{A}(\nabla \tilde{u}_k))\cdot \nabla (\tilde{u}_j-\tilde{u}_k)\bigl](2A - \max(u_j, u_k))_+ h^p dx.
$$
 Define test functions $f$ and $g$ by
\begin{equation}\label{fgdef}f= (\mu - (\tilde{u}_j -\tilde{u}_k)_+)_+, \text{ and } g = (2A - \max(\tilde{u}_k,\tilde{u}_j))_+.
\end{equation}
Notice that $0\leq f\leq \mu$ and $0\leq g\leq 2A$.  Moreover, the chain rule for Sobolev functions (see for example \cite{AH96}, Theorem 3.3.1) guarantees that $f$ and $g$ are in the class $L^{1,p}(B_{2r})$, and satisfy the following properties:
\begin{enumerate}
\item $\nabla f = - \displaystyle\chi_{\{0<\tilde{u}_j-\tilde{u}_k<\mu\}}\nabla (\tilde{u}_j - \tilde{u}_k)$ a.e. on $B_{2r}$,
\item $\nabla g = -\chi_{\{\max(\tilde{u}_j,\tilde{u}_k)<2A\}}[ \chi_{\{\tilde{u}_k-\tilde{u}_j>0\}}\nabla(\tilde{u}_k-\tilde{u}_j)+\nabla \tilde{u}_j]$ a.e. on $B_{2r}$.
\end{enumerate}
To see the second identity write $\max(\tilde{u}_k,\tilde{u}_j) = \max(\tilde{u}_k-\tilde{u}_j,0)+\tilde{u}_j$.  Also, note that the product $fgh^p\in L^{\infty}(B_{2r})\cap L^{1,p}_0(B_{2r})$ (recall $h\in C^{\infty}_0(B_{2r})$), and hence is a valid test function for (\ref{approxsol}).

Using the monotonicity assumption once again, we observe that $-\bigl[(\mathcal{A}(\nabla \tilde{u}_j) -\mathcal{A}(\nabla \tilde{u}_k))\cdot \nabla f \bigl] gh^p\geq 0$ a.e. on $B_{2r}$, and hence
$$IV\leq - \frac{c(\delta, A)}{A}\int_{\Omega_1} \Bigl[(\mathcal{A}(\nabla \tilde{u}_j) -\mathcal{A}(\nabla \tilde{u}_k))\cdot \nabla f \Bigl] gh^p dx.
$$
We denote
$$V = \int_{\Omega_1} \Bigl[(\mathcal{A}(\nabla \tilde{u}_j) -\mathcal{A}(\nabla \tilde{u}_k))\cdot \nabla f \Bigl] gh^p dx,$$
and we will estimate this term by appealing to the PDE (\ref{approxsol}).  In preparation for this, we write
$$V =  -VI - VII + VIII+IX, \text{ with}
$$
$$VI = \int_{\Omega_1}\Bigl[(\mathcal{A}(\nabla \tilde{u}_j) -\mathcal{A}(\nabla \tilde{u}_k))\cdot \nabla g \Bigl] fh^p dx,
$$
$$VII = p \int_{\Omega_1}\Bigl[(\mathcal{A}(\nabla \tilde{u}_j) -\mathcal{A}(\nabla \tilde{u}_k))\cdot \nabla h \Bigl] fgh^{p-1} dx,
$$
$$VIII= \int_{\Omega_1}(\mathcal{A}_j(\nabla \tilde{u}_j) -\mathcal{A}_k(\nabla \tilde{u}_k))\cdot \nabla (fgh^{p}) dx,
$$
\begin{equation}\begin{split}\nonumber
IX = \int_{\Omega_1}\Bigl[&\mathcal{A}(\nabla \tilde{u}_j) - \mathcal{A}_j(\nabla \tilde{u}_j) \\
& + \mathcal{A}_k(\nabla \tilde{u}_k))-\mathcal{A}(\nabla \tilde{u}_k)\Bigl]\cdot \nabla (fgh^{p}) dx.
\end{split}\end{equation}
It is the term $VIII$ which requires care, and it is here where we shall use (\ref{approxsol}).  In all our estimates, we shall make frequent use of the two bounds in (\ref{omega1uniformbd}), and we recall the constant $K$ from those inequalities.

The terms $VI$ and $VII$ can be estimated in a straightforward manner.  For $VI$, observe that $0\leq f\leq \mu$, so we have
\begin{equation}\begin{split}|VI| &\leq \mu\int_{B_{2r}}[|\mathcal{A}(\nabla \tilde{u}_j)|+|\mathcal{A}(\nabla \tilde{u}_j)|]|\nabla g|h^p dx\\
&\leq M\mu\int_{B_{2r}}[|\nabla \tilde{u}_j|^{p-1}+|\nabla \tilde{u}_k|^{p-1}](|\nabla \tilde{u}_j|+|\nabla \tilde{u}_k|) h^p dx.\end{split}\end{equation}
where the second inequality follows from (\ref{ellip}).  Young's inequality now yields
$$|VI| \leq C\mu\Bigl(\int_{B_{2r}}|\nabla \tilde{u}_j|^p h^p dx + \int_{B_{2r}}|\nabla \tilde{u}_k|^p h^p dx\Bigl)\leq C\mu K,
$$
where (\ref{omega1uniformbd}) has been used.  For the estimate of $VII$, we further use that $0\leq g\leq 2A$ and $|\nabla h|\leq C$, and by similar estimates we obtain
$$|VII| \leq C\mu A \Bigl(\int_{B_{2r}}|\nabla \tilde{u}_j|^p h^p dx + \int_{B_{2r}}|\nabla \tilde{u}_k|^p h^p dx\Bigl)^{\frac{p-1}{p}} \leq CA\mu K^{\frac{p-1}{p}}.
$$

For $IX$, we use the continuity of the operator.  Indeed, using the product rule, along with our estimates for $f$ and $g$, we see that $$|\nabla (fgh^p)|\leq C(A+\mu)(|\nabla \tilde{u}_j|+|\nabla{u}_k|)|h|^p + C\mu A|h|^{p-1},$$ and hence we obtain
\begin{equation}\begin{split}\nonumber&\int_{\Omega_1}\Bigl[ \mathcal{A}(\nabla \tilde{u}_j)  - \mathcal{A}_j(\nabla \tilde{u}_j)]\cdot\nabla(fgh^p)dx\\
& = \int_{\Omega_1} \Bigl[\int_{B(x,\varepsilon_j)}\varphi_{\varepsilon_j}(y)( \mathcal{A}(x,\nabla \tilde{u}_j) - \mathcal{A}(x+y,\nabla \tilde{u}_j))dy\Bigl]\cdot \nabla (fgh^p)dx\\
& \leq \omega(\varepsilon_j) C(A^2+\mu)\int_{B_{2r}} |\nabla \tilde{u}_j|^{p-1}\Bigl[|\nabla \tilde{u}_j|+|\nabla \tilde{u}_k|+C\Bigl] dx \leq C(A)K\omega(\varepsilon_j).
\end{split}\end{equation}
(Recall that $\mu<A/2$ and $A>1$).  The right hand side here is of the order $o(1)$ as $j\rightarrow \infty$.  Estimating the difference with $j$ replaced by $k$ in the same manner, we obtain
$$IX = C(A)o(1),\text{ as }j,k\rightarrow \infty.
$$
When compared to (\ref{convmeasgoal}), these estimates for $VI, VII$ and $IX$ are good.

To handle the remaining term $VIII$, we use the equation (\ref{approxsol}) to obtain
$$VIII =  \int_{\Omega_1} fgh^p (\tilde{u}_j^{p-1}\sigma_j -\tilde{u}_k^{p-1}\sigma_k)dx,$$
where $\sigma_j = \varphi_{\varepsilon_{j,1}}*\sigma$.
To continue our estimates we need to make use of the local dual Sobolev property of $\sigma$.  There exists $\vec T\in L^{p'}(B_{2r})^n$ so that $\sigma  = \text{ div }\vec T$ in $\mathcal{D}'(B_{2r})$. As a result, we have
$\sigma_j = \text{div}(\vec T_j)$ with $\vec T_j = \varphi_{\varepsilon_{j}}*T$, and Minkowski's inequality for integrals yields the bound
$||T_j||_{L^{p'}(B_{2r})}  \leq ||T||_{L^{p'}(B_{2r})}.$

Integrating by parts, we proceed by writing $VIII= X + XI + XII,$
with
\begin{equation}\begin{split}\nonumber X = & \int_{\Omega_1} (\tilde{u}_j^{p-1} \vec T_j - \tilde{u}_k^{p-1}\vec T_k)\cdot (\nabla g ) f h^pdx\\
& + p\int_{\Omega_1} (\tilde{u}_j^{p-1}\vec T_j - \tilde{u}_k^{p-1}\vec T_k)\cdot (\nabla h ) h^{p-1}f g dx,
\end{split}\end{equation}
$$XI =  \int_{\Omega_1} (\tilde{u}_j^{p-1}\vec T_j - \tilde{u}_k^{p-1}\vec T_k)\cdot (\nabla f ) g h^pdx,
$$
and
$$XII = \int_{\Omega_1} (\nabla(\tilde{u}_j^{p-1})\cdot\vec T_j - \nabla (\tilde{u}_k^{p-1})\cdot\vec T_k) fg h^p dx.$$
The estimate for $XI$ will be the most delicate (when the gradient falls on $f$).  To bound $X$, recall that $f\leq \mu$ and  $\max(\tilde{u}_j,\tilde{u}_k)\leq 2A$ if  $\nabla g\neq 0$. We therefore see that
$$|X|\leq (2A)^{p-1}\mu\int_{B_{2r}}[|\vec T_j|+|\vec T_k|]|\nabla g| h^p dx \leq C(2A)^{p-1}||\vec T||_{L^{p'}}(B_{2r})K^{1/p}.
$$
The estimate for $XII$ is similar.  Indeed, we notice that
$$|XII| \leq 2A \mu \int_{B_{2r}} [|\nabla \tilde{u}_j^{p-1}|+|\nabla \tilde{u}_k^{p-1}|](|\vec T_j|+|\vec T_k|)h^p dx,$$
which does not exceed $CA\mu||\vec T||_{L^{p'}(B_{2r})}K^{1/p}$.

It remains to estimate $XI$.  It will be convenient to denote
$$F = \{0< \tilde{u}_j -\tilde{u}_k < \mu\}\cap \{\tilde{u}_j \leq 2A\}\cap  B_{2r}.
$$
Note that $\nabla f = 0$ almost everywhere outside of $\{0< \tilde{u}_j -\tilde{u}_k < \mu\}$, and $g=0$ on the set $\max(\tilde{u}_j, \tilde{u}_k)>2A$.  As a result $g \nabla f = 0$ almost everywhere outside $F$, and the integral in $XI$ can be taken over the set $F$.  The triangle inequality now yields
\begin{equation}\nonumber\begin{split}|XI| \leq &\Bigl| \int_{F} \nabla(\tilde{u}_j-\tilde{u}_k)\cdot \vec T_j(\tilde{u}_j^{p-1} - \tilde{u}_k^{p-1})(2A - \tilde{u}_j) h^p dx\Bigl|\\
&+ \Bigl|\int_{F} \nabla(\tilde{u}_j-\tilde{u}_k)\cdot (\vec T_j - \vec T_k) \tilde{u}_k^{p-1}(2A-\tilde{u}_j) h^p dx\Bigl|.
\end{split}\end{equation}
The second term here is easily estimated using the gradient estimates.  Indeed, since $0\leq \tilde{u}_k\leq \tilde{u}_j\leq 2A$ on $F$, we have
\begin{equation}\nonumber\begin{split}\Bigl|\int_{F} & \nabla(\tilde{u}_j-\tilde{u}_k)\cdot (\vec T_j - \vec T_k) \tilde{u}_k^{p-1}(2A-\tilde{u}_j) h^p dx\Bigl| \\
&\leq A^p\int_{F}\bigl[|\nabla \tilde{u}_j|+|\nabla \tilde{u}_k\bigl]|\vec T_j - \vec T_k| h^p dx\\
& \leq CA^p\bigl[||(\nabla \tilde{u}_j) h^p||_p+||(\nabla \tilde{u}_k )h^p||_p\bigl]||(\vec T_j - \vec T_k)h^p||_{p'}\\
&\leq CK^{1/p}A^p||(\vec T_j - \vec T_k)h^p||_{p'},
\end{split}\end{equation}
and from standard properties of the mollification, the right hand side of this bound is of the order
$C(A)o(1)$,  as $j,k\rightarrow \infty.$

Now for our final estimate. We have to find a bound for the integral
$$XIII = \Bigl| \int_{F} \nabla(\tilde{u}_j-\tilde{u}_k)\cdot \vec T_j(\tilde{u}_j^{p-1} - \tilde{u}_k^{p-1})(2A - \tilde{u}_j) h^p dx\Bigl|.
$$
To do this, let $x\in F$, and first note that if $1<p<2$ we have
$$\tilde{u}_j(x)^{p-1}-\tilde{u}_k(x)^{p-1}\leq (\tilde{u}_j(x) - \tilde{u}_k(x))^{p-1}\leq \mu ^{p-1}.
$$
If $p\geq 2$, we instead observe that
\begin{equation}\begin{split}\nonumber \tilde{u}_j(x)^{p-1}-\tilde{u}_k(x)^{p-1}&\leq (p-1)(\tilde{u}_j(x) -\tilde{u}_k(x))\cdot(\tilde{u}_j(x)^{p-2}+\tilde{u}_k(x)^{p-2})\\
&\leq C(p-1)A^{p-2}\mu.
\end{split}\end{equation}
Either way, we obtain
\begin{equation}\begin{split}XIII \leq & C A^{1+ \max(p-2,0)}\mu^{\min(p-1,1)}\int_F |\nabla (\tilde{u}_j-\tilde{u}_k)|\vec T_j||h^p| dx\\
&\leq CK^{1/p}||\vec T||_{L^{p'}(B_{2r})}A^{1+ \max(p-2,0)}\mu^{\min(p-1,1)}.
\end{split}\end{equation}
Bringing all our estimates together, the desired inequality (\ref{convmeasgoal}) follows.
\end{proof}

\subsection{Proof of Theorem \ref{mainthm}, part (ii)}

\begin{proof}[Proof of Theorem \ref{mainthm}, part (ii).]  Suppose there exists a solution $v\in L^{1,p}_{\text{loc}}(\Omega)$ of (\ref{genric}).  Then testing the weak formulation of (\ref{genric}) with $|h|^p$, for $h\in C^{\infty}_0(\Omega)$, we see that
\begin{equation}\nonumber
\langle \sigma, |h|^p \rangle \leq Mp\int_{\Omega} |\nabla v|^{p-1}|\nabla h||h|^{p-1}dx - m(p-1)\int_{\Omega} |\nabla v|^p |h|^p dx,
\end{equation}
where (\ref{ellip}) has been used.  Applying Young's inequality, we have
$$Mp\int_{\Omega} |\nabla v|^{p-1}|\nabla h||h|^{p-1}dx\leq \frac{M^p}{m^{p-1}}\int_{\Omega}|\nabla h|^p + m(p-1)\int_{\Omega} |\nabla v|^p |h|^p dx,
$$
and hence,
$$\langle \sigma, |h|^p \rangle  \leq \frac{M^p}{m^{p-1}}\int_{\Omega}|\nabla h|^p dx.
$$
Using ellipticity of $\mathcal{A}$ (see (\ref{ellip})), we conclude that (\ref{pformupper}) holds with $\lambda = (M/m)^p$.

Let us now suppose in addition that $v$ satisfies (\ref{mult2}) with a constant $C_0>0$.  Testing (\ref{genric}) again with $|h|^p$ for $h\in C^{\infty}_0(\Omega)$, we can estimate
\begin{equation}\begin{split}\langle \sigma, |h|^p \rangle & \geq -pM\int_{\Omega} |\nabla v|^{p-1}|\nabla h||h|^{p-1}dx - M\int_{\Omega} |\nabla v|^p |h|^p dx\\
& \geq - 2M\int_{\Omega}|\nabla v|^p h^p dx -\int_{\Omega} |\nabla h|^p dx.
\end{split}\end{equation}
Where the first inequality here follows from (\ref{ellip}), and the second is the a consequence of Young's inequality.  Applying (\ref{mult2}) we conclude that
$$\langle \sigma, |h|^p \rangle  \geq -(2MC_0 + 1) \int_{\Omega}|\nabla h|^p dx.
$$
Hence (\ref{pformlower}) holds with $\Lambda = M(2MC_0+1)$.
\end{proof}

\section{A remark on higher integrability}\label{highintrem}

In this section we remark on higher integrability of positive solutions of (\ref{genschro}).  We show how the method of Br\'{e}zis and Kato \cite{BK79} can be incorporated into our framework. Let $\Omega\subset\mathbf{R}^n$ be an open set.

\begin{thm}\label{highintthm}  Suppose that $\sigma\in L^{-1,p}_{\rm{loc}}(\Omega)$ satisfies (\ref{pformupper}) with constant $\lambda>0$ and  (\ref{pformlower}) for $\Lambda>0$.   For each $q\in (0,\infty)$, there exists $\lambda(q)>0$ such that if $\lambda< \lambda(q)$, then there exists a positive solution $u\in L^{1,p}_{\rm{loc}}(\Omega)\cap L^q_{\rm{loc}}(\Omega)$ of (\ref{genschro}).
\end{thm}

In dimensions $n=1,2$, the result follows from Theorem \ref{mainthm} using standard Sobolev inequalities.  We shall therefore assume that $n\geq3$.  We will continue to use the notation from the proof of Theorem \ref{mainthm} from Section \ref{equthmsec}.  In particular, we will assume without loss of generality that $\Omega$ is connected, and we will use the approximate sequence of solutions constructed from (\ref{approxsol}).  The result is based on an iterative use of the following lemma:

\begin{lem}\label{highercacc}  Let $s>p$, and suppose that
\begin{equation}\label{lambdaq}\lambda < \lambda(s) = (s-p+1)\Bigl(\frac{p}{s}\Bigl)^p.\end{equation}
Then there exists a constant $C=C(\lambda)$, such that for all $k>j$
\begin{equation}\label{highcacst}
\int_{\Omega_j} |\nabla (u_k)^{s/p}|^p |h|^p dx \leq C\int_{\Omega} u_k^{s} |\nabla h|^p dx, \text{ for all }h\in C^{\infty}_0(\Omega_j).
\end{equation}
\end{lem}

\begin{proof}  The proof mimics the proof of Lemma \ref{est2}.  We leave the details to the reader.
\end{proof}

\begin{proof}[Proof of Theorem \ref{highintthm}] Fix $k>j$.  We may assume that $q>np/(n-p)$ (otherwise the result has already been proved). We put $s_j = \bigl(\frac{n-p}{n}\bigl)^jq,$
for $j=0,\dots, N$.  Here $N$ is chosen to be the largest integer so that $s_N>p$.  Note that $s_N \leq np/(n-p)$.

Let us suppose that $\lambda<\lambda(s_1)$, with $\lambda(s_1)$ as defined in (\ref{lambdaq}).  Since $\lambda(s)$ is monotone decreasing in $s$ for $s>p$, we have $\lambda<\lambda(s_j)$ for all $1\leq j\leq N$.

For each $\ell=0, \dots, N-1$, applying the Sobolev inequality in (\ref{highcacst}) yields the inequality
\begin{equation}\label{sellsob}\Bigl(\int_{\Omega_j} u_k^{s_{\ell}}|h|^p dx\Bigl)^{\frac{n-p}{n}} \leq C\int_{\Omega_j} u_k^{s_{\ell+1}}|\nabla h|^p dx, \text{ for any } h\in C^{\infty}_0(\Omega).
\end{equation}
Now fix a ball $B(x,8r)\subset \Omega_j$, and  define functions $h_{\ell}$, for $\ell=0 \dots N-1,$ satisfying
$$h_{\ell}\in C^{\infty}_0(B(x, (1+ \tfrac{\ell+1}{N})r)),\, h_{\ell} \equiv 1 \text{ on }B(x, (1+\tfrac{\ell}{N})r), \, |\nabla h_{\ell}| \leq \tfrac{CN}{r}.
$$
Substituting these test functions in (\ref{sellsob}) yields
$$\Bigl(\dashint_{B(x,(1+\ell/N)r)} u_k^{s_{\ell}} dx\Bigl)^{\frac{n-p}{n}}\leq CN^p \dashint_{B(x, (1+(\ell+1)/N)r)} u_k^{s_{\ell+1}} dx, \text{ for each }\ell.
$$
An $N$-fold iteration of the preceeding inequality results in
$$\dashint_{B(x,r)} u_k^{s_0} dx \leq C(N, q, \lambda, r)\Bigl(\dashint_{B(x,2r)} u_k^{s_N} dx\Bigl)^{\frac{nN}{n-p}}.$$
Since $s_N\leq np/(n-p)$, the right hand side of this equation can be bounded using the estimate (\ref{apriori1}) (by way of Sobolev's inequality).  We arrive at
\begin{equation}\label{unfqint}\int_{B(x,r)} u_k^q dx \leq C(q, B(x,r),\lambda, \Lambda, \Omega_j, B).
\end{equation}
Mimicking the passage to the limit in Section \ref{passtolimit}, we arrive (with an additional application of Fatou's lemma) at a positive solution $u$ of (\ref{genschro}) with the  property that $u\in L^q_{\text{loc}}(\Omega)$.
\end{proof}


\section{The proof of Theorem \ref{sobcor}}\label{sobthmsec}

For a measure $\mu$ and $0<\alpha<n$, define the Riesz potential of order $\alpha$, $\mathbf{I}_{\alpha}(\mu)(x) = \int_{\mathbf{R}^n} \frac{d\mu(y)}{|x-y|^{n-\alpha}}.
$
Denote by $(-\Delta)^{-1}$ the Green's operator in $\mathbf{R}^n$, given by
\begin{equation}
(-\Delta)^{-1}(\mu)(x) = \begin{cases}
\;\displaystyle \frac{1}{2\pi}\int_{\mathbf{R}^n} \log|x-y| d\mu(y), \; \text{ if }n=2,\\
\;c_n\mathbf{I}_2(\mu)(x),\; \text{ if }n\geq 3.
\end{cases}
\end{equation}
where $c_n>0$ has been chosen so that $-c_n\Delta |\cdot-y|^{2-n}  = \delta_{y} \text{ in }\mathcal{D}'(\mathbf{R}^n).$
Here $\delta_{y}$ is the Dirac delta measure concentrated at the point $y$.

\begin{proof}[Proof of Theorem \ref{sobcor}]
We shall first prove part (i).  The sufficiency of the representation of $\sigma$ as $\sigma = \text{div}(\vec{\Gamma})$, with $\vec\Gamma$ satisfying (\ref{positiveineq}) follows from H\"{o}lder's inequality.  On the other hand, suppose $\sigma$ satisfies (\ref{pformbd}) with a constant $C_0>0$.  Then note that $\tilde\sigma = \frac{(p-1)^{2-p}}{2C_0} \sigma$
satisfies the hypothesis of Theorem \ref{introthmequ}.

Applying part (i) of Theorem \ref{introthmequ}, we see that there exists $v\in L^{1,p}_{\text{loc}}(\Omega)$ satisfying
$-\text{div}(|\nabla v|^{p-2}\nabla v) = |\nabla v|^p + \tilde \sigma$ in $\mathbf{R}^n$,
such that
\begin{equation}\label{nablavprop}\int_{\mathbf{R}^n}|\nabla v|^p h^p dx\leq C\int_{\mathbf{R}^n}|\nabla h|^p, \text{ for all }h\in C^{\infty}_0(\mathbf{R}^n).
\end{equation}
Now denote $d\mu  = |\nabla v|^p dx $.  Then $\mu$ satisfies
$$\int_{\mathbf{R}^n}|h|^p d\mu \leq C\int_{\mathbf{R}^n}|\nabla h|^p dx, \text{ for all }h\in C^{\infty}_0(\mathbf{R}^n).
$$
It now follows from \cite{MV95} (see also Theorem 1.7 of \cite{V}) that there exists a constant $C>0$ such that
\begin{equation}\label{mucap}\int_E (\mathbf{I}_1(\mu))^{p'} dx \leq C\text{cap}_p(E), \text{ for all compact sets }E\subset \mathbf{R}^n.
\end{equation}
We claim that there exists a solution $w$ of
\begin{equation}\label{wexist}-\Delta w =  \Bigl(\frac{2C_0}{(p-1)^{2-p}}\Bigl)\mu= \Bigl(\frac{2C_0}{(p-1)^{2-p}}\Bigl)|\nabla v|^p\text{ in }\mathbf{R}^n,\end{equation}
along with a constant $C= C(C_0)$ such that
\begin{equation}\label{wgradest}\int_E |\nabla w|^{p'} dx \leq C\text{cap}_p(E), \text{ for all compact sets }E\subset \mathbf{R}^n.
\end{equation}
To see this, let $\mu_N = |\nabla v|^p\chi_{B(0,2^N)} dx$.  Then (\ref{mucap}) is satisfied with $\mu$ replaced by $\mu_N$.  Let
$$w_N  = \frac{2C_0}{(p-1)^{2-p}} \Delta^{-1}\mu_N - c_N,
$$
where $c_N$ is chosen to ensure that $|\int_{B(0,1)}w_N dx| = 1.$

Using the inequality $|\nabla \Delta^{-1}\mu_N|\leq c\mathbf{I}_1(\mu_N)$, we see that
\begin{equation}\label{wNcap}\int_E |\nabla w_N|^{p'} dx \leq C(C_0) \text{cap}_p(E), \text{ for all compact sets }E\subset \mathbf{R}^n.
\end{equation}
Therefore the sequence $(w_N)_N$ is uniformly bounded in $L^{1,p'}_{\text{loc}}(\mathbf{R}^n)$.  By weak compactness and a diagonal argument, there is a subsequence of $w_N$ (still denoted by $w_N$), so that $w_N$ converges weakly to $w$ in $L^{1,p'}_{\text{loc}}(\mathbf{R}^n)$.  Using Rellich's theorem, and the normalization on $w_N$, we see that $w$ is not infinite.  This limit function $w$ is easily seen to be a distributional solution of (\ref{wexist}) satisfying (\ref{wgradest}).

Notice that the inequality (\ref{wgradest}) is equivalent to (see  \cite{Maz85}, Sec. 2.3.4, p. 160)
\begin{equation}\label{nablawprop}\int_{\mathbf{R}^n} |\nabla w|^{p'}|h|^p\leq C(C_0)\int_{\mathbf{R}^n} |\nabla h|^p dx, \, \, \text{for all} \, \, h \in C^\infty_0(\mathbf{R}^n).
\end{equation}
Let
$\vec\Gamma = - \bigl(\frac{2C_0}{(p-1)^{2-p}}\bigl) |\nabla v|^{p-2}\nabla v + \nabla w.$
From displays (\ref{nablavprop}) and (\ref{nablawprop}), we see that $\vec\Gamma$ satisfies the conclusion of the theorem.

Let us now turn to part (ii), which is more straightforward.  We suppose $p\geq n$.  As in the proof of part (i), we can reduce matters to when $C_0<(p-1)^{2-p}$ in (\ref{pformbdentire}).  Applying Theorem \ref{introthmequ}, we deduce the exists of $v\in L^{1,p}_{\text{loc}}(\mathbf{R}^n)$, such that
\begin{equation}\label{riclou}-\text{div}(|\nabla v|^{p-2}\nabla v) = |\nabla v|^p + \sigma \text{ in }\mathbf{R}^n,
\end{equation}
satisfying (\ref{nablavprop}).  It is immediate from (\ref{nablavprop}) and from the definition of capacity (\ref{cap}) that
$$\int_{E} |\nabla v|^p dx \leq C \text{cap}_p(E), \text{ for all compact sets }E\subset \mathbf{R}^n.
$$
However, with $p\geq n$, it is well known (see \cite{Maz85}, Sec. 2.2.4, p. 148) that $\text{cap}_p(E) = 0$ for all compact sets $E\subset\mathbf{R}^n.$
Therefore $|\nabla v|\equiv 0$, and hence $\sigma \equiv 0$.
\end{proof}


\begin{thebibliography}{DMMOP99}

\bibitem[ADP06]{ADP06} B. Abdellaoui, A. Dall'Aglio and I. Peral, \emph{Some remarks on elliptic problems with critical growth in the gradient}, J. Diff. Equations, \textbf{222} (2006), 21--62.


\bibitem[AHBV09]{AHBV09} H. A. Hamid and M. F.  Bidaut-Veron, \emph{On the connection between two quasilinear elliptic problems with source terms of order 0 or 1,} Commun. Contemp. Math. \textbf{12} (2010), 727--788.

\bibitem[Ad09]{Ad09} D. R. Adams, \emph{My love affair with the Sobolev inequality,} Sobolev Spaces in Mathematics, I. 
Int. Math. Ser.  \textbf{8}, 1--23, Springer, New York, 2009.

\bibitem[AH96]{AH96}
D.~R. Adams and L.~I. Hedberg, \emph{Function Spaces and Potential Theory},
  Grundlehren der math. Wissenschaften \textbf{314}, Springer, 1996.


\bibitem[AFT04]{AFT04} B. Alziary, J. Fleckinger and P. Tak\'{a}c, \emph{Variational methods for a resonant problem with the $p$-Laplacian in $R^N$}, Electron. J. Diff. Equations \textbf{76}
(2004), 1--32.




\bibitem[BBGPV95]{BBGVP}
P.~B\'{e}nilan, L.~Boccardo, R.~Gariepy, M.~Pierre, and J.~Vazquez, \emph{An $L^1$ theory of exsitence and uniqueness of solutions of nonlinear elliptic equations}, Ann. Scuola Norm. Sup.  Pisa \textbf{22} (1995), 241--273.

\bibitem[Bre11]{Bre11}
H. Br\'{e}zis, \emph{Functional Analysis, Sobolev Spaces and Partial Differential Equations,}
Universitext, Springer, New York, 2011.


\bibitem[BK79]{BK79}
H. Br\'{e}zis and T. Kato, \emph{Remarks on the Schr\"{o}dinger operator with singular complex potentials,}
J. Math. Pures Appl. (9) \textbf{58} (1979),  137--151.



\bibitem[BM97]{BM97}H. Br\'{e}zis and M. Marcus, \emph{Hardy's inequalities revisited,}  Ann. Scuola Norm. Sup. Pisa Cl. Sci. (4) \textbf{25} (1997),  217--237.


\bibitem[BN83]{BN83} H. Br\'{e}zis and L. Nirenberg, \emph{Positive solutions of nonlinear elliptic equations involving critical Sobolev exponents,}
Comm. Pure Appl. Math. \textbf{36} (1983),  437--477.

\bibitem[DMMOP]{DMMOP}
G.~Dal Maso, F.~Murat, L.~Orsina, and A.~Prignet, \emph{Renormalized solutions of elliptic equations with general measure data},
Ann. Scuola Norm. Sup. Pisa Cl. Sci. \textbf{28} (1999), 741--808.


\bibitem[EE87]{EE87} D. E. Edmunds and W. D Evans,  \emph{Spectral Theory and Differential Operators,} Oxford Mathematical Monographs, The Clarendon Press, Oxford University Press, New York, 1987.



\bibitem[FM00]{FM00}
V.~Ferone and F.~Murat, \emph{Nonlinear problems having natural growth in the
  gradient: an existence result when the source terms are small}, Nonlinear
  Analysis \textbf{42} (2000), 1309--1326.



\bibitem[Giu03]{Giu03}
E.~Giusti, \emph{Direct Methods in the Calculus of Variations}, World Scientific, River Edge, NJ, 2003.


\bibitem[GT03]{GT03}
N. Grenon and C. Trombetti, \emph{Existence results for a class of nonlinear elliptic problems with p-growth in the gradient,} Nonlinear Anal. \textbf{52} (2003), 931--942.


\bibitem[HKM06]{HKM}
J.~Heinonen, T.~Kilpel\"{a}inen, and O.~Martio, \emph{Nonlinear Potential
 Theory of Degenerate Elliptic Equations}, Dover Publications, 2006 (unabridged republ. of 1993 edition, Oxford University Press).

\bibitem[JMV11]{JMV11}
B. J. Jaye, V. G. Maz'ya, and I. E. Verbitsky, \emph{Existence and regularity of positive solutions to elliptic equations of Schr\"{o}dinger type}, J. d'Analyse Math. (to appear), arXiv/1103.0698.



\bibitem[KKT11]{KKT11} T. Kilpel\"{a}inen, T. Kuusi, and A. Tuhola-Kujanp\"{a}\"{a},  \emph{Superharmonic functions are locally renormalized solutions}, Ann. Inst. Henri Poincar\'{e}, Anal. Non 
Lin\'eare \textbf{28} (2011), 775--795.

\bibitem[LL01]{LL01}
E. Lieb and M. Loss, \emph{Analysis,}  Second ed.,  Graduate Studies  Math., \textbf{14}, Amer. Math. Soc., Providence, RI, 2001.






\bibitem[MZ97]{MZ97}
J.~Maly and W.~Ziemer, \emph{Fine Regularity of Solutions of Elliptic Partial
  Differential Equations}, Math. Surveys and Monographs \textbf{51},
 AMS, Providence, RI, 1997.

\bibitem[MS00]{MS00}
M. Marcus and I. Shafrir, \emph{An eigenvalue problem related to Hardy's $L^p$ inequality.} Ann. Scuola Norm. Super. Pisa  \textbf{29} (2000), 581--604.

\bibitem[Maz11]{Maz85}
V. G. Maz'ya, \emph{Sobolev Spaces with Applications to Elliptic Partial Differential Equations,} Second, revised and augmented ed., Grundlehren der math. Wissenschaften, \textbf{342}, Springer, Heidelberg, 2011.

\bibitem[MV95]{MV95} V. G. Maz'ya and I. E. Verbitsky, \emph{Capacitary inequalities for fractional integrals, with applications to partial differential equations and Sobolev multipliers,} Ark. Mat. \textbf{33} (1995), 81--115.

\bibitem[MV02a]  {MV1} V.~G. Maz'ya and I.~E. Verbitsky,
{\em The Schr\"odinger operator on the energy space: boundedness
and compactness criteria},  Acta Math. {\bf 188} (2002),
263--302.

\bibitem[MV02b]  {MV2} V.~G. Maz'ya and I.~E. Verbitsky,
{\em Boundedness and compactness
criteria for the one-dimensional Schr\"odinger operator},
 Function Spaces, Interpolation Theory and Related Topics.
Proc. Jaak Peetre Conf.,
Lund, Sweden, August 17-22, 2000, Eds. M. Cwikel, A. Kufner, and
G. Sparr, De Gruyter, Berlin, 2002, 369--382.

\bibitem[MV06]  {MV5} V.~G. Maz'ya and I.~E. Verbitsky,
{\em Form boundedness of the general second order differential
operator}, Comm. Pure Appl. Math. \textbf{59} (2006),  1286--1329.

\bibitem[Min07]{Min07}
G.~Mingione, \emph{The Calder\'{o}n-Zygmund theory for elliptic problems with measure data},
Ann. Scuola Norm. Super. Pisa \textbf{6} (2007),  195--261.


\bibitem[MP02]{MP02}
F. Murat and A. Porretta, \emph{Stability properties, existence, and nonexistence of renormalized solutions for elliptic equations with measure data,} Comm. Partial Diff. Equations \textbf{27} (2002), 2267--2310.


\bibitem[Pin07]{Pin07}
Y. Pinchover, \emph{Topics in the theory of positive solutions of second-order elliptic and parabolic partial differential equations,}  Spectral Theory and Mathematical Physics: A Festschrift in Honor of Barry Simon's 60th Birthday,  Eds. F. Gesztesy et al., Proc. Symp. Pure Math. \textbf{76}, Amer. Math. Soc., Providence, RI, 2007, 329--356.


\bibitem[PT07]{PT07} Y. Pinchover and K. Tintarev, \emph{Ground state alternative for $p$-Laplacian with potential term,} Calc. Var. Partial Diff. Equations \textbf{28} (2007), 179--201.


\bibitem[Pol03]{Pol03} A. Poliakovsky,  \emph{On minimization problems which approximate Hardy $L^p$ inequality,}
Nonlinear Anal. \textbf{54} (2003), 1221--1240.


\bibitem[PS05]{PS05}
A. Poliakovsky and I. Shafrir, \emph{Uniqueness of positive solutions for singular problems involving the $p$-Laplacian,}
Proc. Amer. Math. Soc. \textbf{133} (2005), 2549--2557.

\bibitem[Por02]{Por02} A. Porretta, \emph{Nonlinear equations with natural growth terms and measure data,} Proceedings of the 2002 Fez Conference on Partial Differential Equations, 183--202 (electronic), Electron. J. Diff. Equations Conf., \textbf{9}, Southwest Texas State Univ., San Marcos, TX, 2002.


\bibitem[PS06]{PS06} A. Porretta and  S. Segura de Le\'{o}n, \emph{Nonlinear elliptic equations having a gradient term with natural growth,}
J. Math. Pures Appl. (9) \textbf{85} (2006),  465--492.

\bibitem[RS75] {RS75} M. Reed and B. Simon,
{\em Methods of Modern
Mathematical
Physics, II: Fourier Analysis, Self-Adjointness},
Academic Press, New York--London, 1975.




\bibitem[Ser64]{Ser2}
J.~Serrin, \emph{Local behavior of solutions to quasi-linear equations}, Acta
  Math. \textbf{111} (1964), 247--301.


\bibitem[Sme99]{Sme99}
D. Smets, \emph{A concentration-compactness lemma with applications to singular eigenvalue problems,}
J. Funct. Anal. \textbf{167} (1999), 467--480.


\bibitem[St93]{St93}
E. M. Stein,  \emph{Harmonic Analysis: Real-Variable Methods, Orthogonality, and Oscillatory Integrals,} Princeton Math. Ser., \textbf{43}, Monographs in Harmonic Analysis, III, Princeton University Press, Princeton, NJ, 1993.

\bibitem[SW99]{SW99}
A. Szulkin and M. Willem, \emph{Eigenvalue problems with indefinite weight}, Studia Math. \textbf{135} (1999), 191--201.

\bibitem[Ter98]{Ter98}A. Tertikas,  \emph{Critical phenomena in linear elliptic problems,} J. Funct. Anal. \textbf{154} (1998), 42--66.

\bibitem[V] {V} I.~E. Verbitsky,
{\em Nonlinear potentials and trace
inequalities},
The Maz'ya Anniversary
Collection, Eds. J. Rossmann,  P. Tak\'ac, and G. Wildenhain,
Operator Theory: Adv. Appl. {\bf 110} (1999),
 323--343.

\end{thebibliography}
\end{document}